\documentclass[12pt]{article}
 
 \usepackage{amsmath,amssymb,amscd,amsthm,esint}
 
\usepackage{graphics,amsmath,amssymb,amsthm,mathrsfs}

\newtheorem{theorem}{Theorem}[section]

\newtheorem{lemma}[theorem]{Lemma}

\newtheorem{corollary}[theorem]{Corollary}

\newtheorem{remark}[theorem]{Remark}

\def\xxint#1#2#3{{\setbox0=\hbox{$#1{#2#3}{\int}$}
  \vcenter{\hbox{$#2#3$}}\kern-.5\wd0}}

\def \rd {{\mathbb R}^d}

\def\loc{\text{\rm loc}}

\def\varep{\varepsilon}

\newcommand{\average}{-\!\!\!\!\!\!\int}

\begin{document}

\title
{\bf Boundary Estimates  \\ in Elliptic Homogenization}

\author{Zhongwei Shen\thanks{Supported in part by NSF grant DMS-1161154.}}


\date{ }

\maketitle

\begin{abstract}

For a family of systems of linear elasticity with 
rapidly oscillating periodic coefficients, we establish sharp boundary estimates with either Dirichlet or Neumann conditions,
uniform down to the microscopic scale, without smoothness assumptions on the coefficients.
Under additional smoothness conditions, these estimates, combined with the corresponding 
local estimates, lead to the full Rellich type estimates in Lipschitz domains and Lipschitz estimates in $C^{1, \alpha}$ domains.
The $C^\alpha$, $W^{1,p}$, and $L^p$ estimates in $C^1$ domains for systems with VMO coefficients are also studied.
The approach is based on certain estimates on convergence rates.
As a bi-product, we obtain sharp $O(\varep)$ error estimates in $L^q(\Omega)$ for
$q=\frac{2d}{d-1}$ and a Lipschitz domain $\Omega$, with no smoothness assumption on the coefficients.

\bigskip

\noindent{\it MSC2010:} \ \ 35B27, 35J55, 74B05.

\noindent{\it Keywords:} homogenization; systems of elasticity; convergence rates,
Rellich estimates; Lipschitz estimates.

\end{abstract}

\tableofcontents

\section{Introduction}
\setcounter{equation}{0}

The purpose of this paper is to establish sharp boundary estimates with either Dirichlet or Neumann conditions, uniform down to the microscopic scale, for a family of second-order elliptic systems in divergence form with rapidly oscillating coefficients, 
without any smoothness assumption on the coefficients. Under additional smoothness conditions, these estimates,
combined with the corresponding local estimates, lead to the full Rellich type estimates in Lipschitz domains
and Lipschitz estimates in $C^{1, \alpha}$ domains.
The $C^\alpha$, $W^{1,p}$, and $L^p$ estimates in $C^1$ domains for systems with
VMO coefficients are also investigated.
To fix the idea we shall consider the systems of linear elasticity with periodic coefficients in this paper.
We mention that the same results, without the complications introduced by
 rigid displacements,  hold for general second-order elliptic systems with periodic coefficients 
satisfying the stronger ellipticity condition (\ref{s-ellipticity}) (the symmetry condition is also needed
for Rellich estimates in Lipschitz domains).
We further point out that
although we restrict ourself to the periodic case, 
our approach, which is based on certain estimates on convergence rates in $H^1$ and $L^2$,
extends to non-periodic settings, provided  that  the interior correctors or approximate correctors
satisfy certain $L^2$ conditions.
The compactness methods, which were introduced to the study of homogenization in \cite{AL-1987}
and have played an important role in establishing regularity results in the periodic setting (see e.g.  \cite{AL-1987, AL-1989-II,
KLS1, Kenig-Prange-2014}), are not used 
in this paper.  As a bi-product of our new approach, we also obtain sharp $O(\varep)$ error estimates in $L^q(\Omega)$ for
$q=\frac{2d}{d-1}$ and a Lipschitz domain $\Omega$, with no smoothness assumption on the coefficients.

More precisely, consider the systems of linear elasticity
\begin{equation}\label{operator}
\mathcal{L}_\varep =-\text{div} \left(A(x/\varep)\nabla \right)
=-\frac{\partial}{\partial x_i}
\left[ a_{ij}^{\alpha\beta} \left(\frac{x}{\varep}\right)\frac{\partial}{\partial x_j}\right], 
\quad \varep>0.
\end{equation}
We will assume that $ A(y)=\big( a_{ij}^{\alpha\beta} (y)\big)$
with $1\le i,j,\alpha,\beta\le d$ is real and satisfies the elasticity condition
\begin{equation}\label{ellipticity}
\aligned
& a_{ij}^{\alpha\beta} (y) =a_{ji}^{\beta\alpha} (y) =a_{\alpha j}^{i\beta} (y),\\
 &  \kappa_1 |\xi|^2  \le a_{ij}^{\alpha\beta} (y)
\xi_i^\alpha\xi_j^\beta \le \kappa_2 |\xi|^2 
\endaligned
\end{equation}
for $y\in \rd$ and for symmetric matrix $\xi =(\xi_i^\alpha)\in \mathbb{R}^{d\times d}$,
where $\kappa_1, \kappa_2>0$ (the summation convention is used throughout the paper).
We will also assume that $A(y)$ is 1-periodic; i.e.,
\begin{equation}\label{periodicity}
A(y+z)=A(y) \quad \text{ for } y\in \rd \text{ and } z\in \mathbb{Z}^d.
\end{equation}

\begin{theorem}\label{main-theorem-1}
Suppose that $A$ satisfies conditions (\ref{ellipticity})-(\ref{periodicity}).
Let $\Omega$ be a bounded Lipschitz domain in $\rd$.
Let $u_\varep \in H^1(\Omega; \rd)$ be the weak solution to the Dirichlet problem:
\begin{equation}\label{DP-1}
\mathcal{L}_\varep (u_\varep)=F \quad \text{ in }\Omega \quad \text{ and } \quad u_\varep =f \quad 
\text{ on } \partial\Omega,
\end{equation}
where  $F\in L^p(\Omega; \rd)$ for $p=\frac{2d}{d+1}$ and $f\in H^1(\partial\Omega; \rd)$.
Then, for $\varep \le  r< \text{diam}(\Omega)$,
\begin{equation}\label{main-estimate-1}
\left\{ \frac{1}{r}\int_{\Omega_r} |\nabla u_\varep|^2\right\}^{1/2}
\le C \Big\{ \| F\|_{L^p(\Omega)} +\| f\|_{H^1(\partial \Omega)} \Big\},
\end{equation}
where $\Omega_r =\big\{ x\in \Omega: \text{dist} (x, \partial\Omega)<r \big\}$.
The constant $C$ depends only on $d$, $\kappa_1$, $\kappa_2$, and  the Lipchitz character of $\Omega$.
\end{theorem}

Let $\mathcal{R}$ denote the space of rigid displacements,
\begin{equation}\label{R}
\mathcal{R}= \Big\{ Mx+q: 
M^T=-M \in \mathbb{R}^{d\times d}  \text{ and } q\in \rd \Big\},
\end{equation}
where $(Mx)^\alpha =M_i^\alpha x_i$ and $M^T$ denotes the transpose of matrix $M$.
By $u\perp \mathcal{R}$ we mean $u\perp \mathcal{R}$  in $L^2(\Omega; \rd)$, i.e.,
$\int_\Omega u\cdot \phi =0$ for any $\phi\in \mathcal{R}$.
We will use $\frac{\partial u_\varep}{\partial \nu_\varep}$ to denote
the conormal derivative of $u_\varep$ associated with $\mathcal{L}_\varep$.

\begin{theorem}\label{main-theorem-2}
Suppose that $A$ and $\Omega$ satisfy the same conditions as in Theorem \ref{main-theorem-1}.
Let $u_\varep \in H^1(\Omega; \rd)$ be a weak solution to the Neumann problem:
\begin{equation}\label{NP-1}
\mathcal{L}_\varep (u_\varep)=F \quad \text{ in }\Omega \quad \text{ and }
\quad  \frac{\partial u_\varep}{\partial \nu_\varep}
 =g \quad 
\text{ on } \partial\Omega,
\end{equation}
where $F\in L^p(\Omega; \rd)$ for $p=\frac{2d}{d+1}$, $g\in L^2(\partial\Omega; \rd)$ and
$\int_\Omega F +\int_{\partial\Omega} g=0$. 
Also assume that $u_\varep \perp \mathcal{R}$.
Then, for $\varep \le  r< \text{diam}(\Omega)$,
\begin{equation}\label{main-estimate-2}
\left\{ \frac{1}{r}\int_{\Omega_r} |\nabla u_\varep|^2\right\}^{1/2}
\le C \Big\{ \| F\|_{L^p(\Omega)} +\| g\|_{L^2(\partial \Omega)} \Big\},
\end{equation}
where  $C$ depends only on $d$, $\kappa_1$, $\kappa_2$, and  the Lipschitz character of $\Omega$.
\end{theorem}

Estimates (\ref{main-estimate-1}) and (\ref{main-estimate-2}), which are scaling-invariant,
 may be regarded as the Rellich estimates,
uniform down to the scale $\varep$, in Lipschitz domains for the elasticity operators $\mathcal{L}_\varep$.
Indeed, if the coefficient matrix $A$ is constant, then (\ref{main-estimate-1}) and (\ref{main-estimate-2})
hold for any $0<r<\text{diam}(\Omega)$. Suppose that $F=0$.
By letting $r\to 0$, one recovers the full Rellich estimates in Lipschitz domains,
\begin{equation}\label{Rellich-1}
\| \nabla u_\varep \|_{L^2(\partial\Omega)} \le C\, \|u_\varep \|_{H^1(\partial\Omega)}
\quad \text{ and } \quad \| \nabla u_\varep \|_{L^2(\partial\Omega)}
\le C\, \big\| \frac{\partial u_\varep}{\partial\nu_\varep}\big\|_{L^2(\partial\Omega)},
\end{equation}
which were proved in \cite{Fabes-1988, Dahlberg-1988} for second-order
elliptic systems with constant coefficients, using integration by parts
(see \cite{Kenig-book} for references on related work on boundary value problems in Lipschitz domains).
We should note that our proof of Theorems \ref{main-theorem-1} and \ref{main-theorem-2}
uses the nontangential maximal function estimates  in \cite{Dahlberg-1988}.
On the other hand, under certain smoothness conditions on $A$, the Rellich estimates
hold for the operator $\mathcal{L}_1$ on Lipschitz domains with
$\text{diam}(\Omega)\le 1$.
By a blow-up argument as well as some localization procedures, this implies that
\begin{equation}\label{Rellich-2}
\aligned
&\| \nabla u_\varep\|_{L^2(\partial\Omega)}
\le C\, \Big\{  \| \nabla_{\tan} u_\varep\|_{L^2(\partial\Omega)} + \varep^{-1/2} \|\nabla u_\varep\|_{L^2(\Omega_\varep)} \Big\},\\
&\| \nabla u_\varep\|_{L^2(\partial\Omega)}
\le C\, \Big\{  \big\| \frac{\partial u_\varep}{\partial \nu_\varep}\big\|_{L^2(\partial\Omega)} 
+ \varep^{-1/2} \|\nabla u_\varep\|_{L^2(\Omega_\varep)} \Big\},
\endaligned
\end{equation}
where $\nabla_{\tan} u_\varep$ denotes the tangential derivative of $u_\varep$ on $\partial\Omega$.
We emphasize that the estimates (\ref{Rellich-2}) are local and structure conditions such as periodicity are not needed.
However, with the additional periodicity condition, one may combine the local estimates (\ref{Rellich-2})
with the estimates in Theorems \ref{main-theorem-1} and \ref{main-theorem-2} to obtain 
the full Rellich estimate (\ref{Rellich-1}), uniform in $\varep$,
 for operators $\mathcal{L}_\varep$ (see Remark \ref{remark-3.1}). Thus we have been able to completely 
 separate the large-scale regularity due to homogenization from the small-scale regularity 
 due to smoothness of the coefficients.
 
Under the periodicity condition and the H\"older continuity condition on $A$,
 the  uniform Rellich estimates (\ref{Rellich-1})
were proved in \cite{Kenig-Shen-1, Kenig-Shen-2} for a family of elliptic operators $\{ \mathcal{L}_\varep\}$,
where $\mathcal{L}_\varep =-\text{div} \big(A(x/\varep)\nabla\big)$ and
$A(y) =\big(a_{ij}^{\alpha\beta} (y)\big)$
with $1\le i,j\le d$ and $1\le \alpha, \beta\le m$ satisfies the ellipticity condition
\begin{equation}\label{s-ellipticity}
\mu |\xi|^2 \le a_{ij}^{\alpha\beta}  (y)\xi_i^\alpha\xi_j^\beta \le \frac{1}{\mu} |\xi|^2
\end{equation}
for $y\in \rd$ and $\xi=(\xi_i^\alpha)\in \mathbb{R}^{d\times m}$ as well as  the symmetry condition $A^*=A$, i.e.,
$a_{ij}^{\alpha\beta} =a_{ji}^{\beta\alpha}$.
The results were used to establish the uniform solvability of
the $L^2$ Dirichlet, regularity, and Neumann problems for the system $\mathcal{L}_\varep (u_\varep)=0$
in Lipschitz domains.
It worths pointing out that the Rellich estimates (\ref{Rellich-1})
are not accessible by compactness methods.
One of the key steps in \cite{Kenig-Shen-1,  Kenig-Shen-2}
 uses integration by parts and relies on the observation
that $\mathcal{L}_1 (Q)=Q(\mathcal{L}_1)$, where
$$
Q(u)(x^\prime, x_d)= u(x^\prime, x_d+1)-u(x^\prime, x_d).
$$
As a result, the approach does not seem to apply if the coefficients are not periodic. We mention that
even with periodic coefficients, the direct extension of the methods used in \cite{Kenig-Shen-1, Kenig-Shen-2}
is problematic for systems of elasticity, due to the weaker ellipticity condition and the lack of (uniform) Korn inequalities 
on boundary layers. 

In this paper we develop a new approach to uniform boundary regularity in quantitative homogenization 
of elliptic equations and systems.
Let $u_0$ denote the solution of the boundary value problem for the homogenized system with the same data.
The basic idea is to consider the function 
\begin{equation}\label{w}
w_\varep =u_\varep -u_0 -\varep \chi_j^\beta (x/\varep) K^2_\varep \Big( \frac{\partial u_0^\beta}
{\partial x_j} \eta_\varep\Big)
\end{equation}
in $\Omega$, where $\chi=\big(\chi_j^\beta\big)$ denotes the matrix of correctors,
$K_\varep$ is a smoothing operator at scale $\varep$, and $\eta_\varep \in C_0^\infty(\Omega)$ is a
cut-off function with support in $\big\{ x\in \Omega: \text{dist}(x, \partial \Omega)\ge 3\varep\big\}$.
Using energy estimates for the operator $\mathcal{L}_\varep$
as well as sharp boundary regularity estimates for $u_0$, we are able to bound
$$
\varep^{-1/2} \| w_\varep\|_{H^1(\Omega)}
$$
by the r.h.s. of estimates (\ref{main-estimate-1}) and (\ref{main-estimate-2}), respectively.
This, together with sharp estimates for $u_0$, yields the desired estimates for
$$
r^{-1/2} \|\nabla u_\varep\|_{L^2(\Omega_r)},
$$
for $\varep\le r< \text{diam}(\Omega)$.
We mention that since $\mathcal{L}_0$ has constant coefficients, the sharp boundary estimates
in Lipschitz domains in terms of nontangential maximal functions are known \cite{Fabes-1988, Dahlberg-1988}.
Also, because of the use of the smoothing operator $K_\varep$, which is motivated by the work \cite{Pas-2006, Suslina-2013} (also see \cite{Griso-2004,Onofrei-2007,KLS2,Suslina-2013-siam}),
we only need to assume that 
$$
\sup_{x\in \rd} \int_{B(x,1)} \Big( |\chi(y)|^2 +|\nabla \chi(y)|^2\Big) \, dy < \infty,
$$
and that a similar estimate holds for a dual corrector $\phi=\big( \phi_{k i j}^{\alpha\beta} \big)$
 (see (\ref{phi}) for its definition).
As such, it is possible to extend the approach to the almost-periodic or other non-periodic settings.
We plan to carry out this study in a separate work.

As we mentioned before, the estimates in Theorems \ref{main-theorem-1} and \ref{main-theorem-2}
may be used to establish uniform solvability of $L^2$ boundary value problems for $\mathcal{L}_\varep$
in Lipschitz domains \cite{Kenig-Shen-1, Kenig-Shen-2, Geng-Shen-Song-2}.
They also can be used to obtain sharp $O(\varep)$ error estimates in $L^q(\Omega)$
for $q=\frac{2d}{d-1}$ and a Lipschitz domain $\Omega$, with no smoothness assumption on the coefficients.

\begin{theorem}\label{main-theorem-L}
Suppose that $A$ and $\Omega$ satisfy the same conditions as in Theorem \ref{main-theorem-1}.
Let $u_\varep$  be a weak solution to (\ref{DP-1}) or (\ref{NP-1}), and $u_0$  the weak solution of the homogenized
system with the same data.
Suppose that $u_0\in H^2(\Omega; \mathbb{R}^d)$.
In the case of the Neumann problem (\ref{NP-1}) we further  assume that
$u_\varep, u_0 \perp \mathcal{R}$.
Then
\begin{equation}\label{sharp-rate}
\| u_\varep - u_0\|_{L^q(\Omega)} \le C\, \varep  \| u_0\|_{H^2(\Omega)},
\end{equation}
where $q=p^\prime =\frac{2d}{d-1}$ and $C$ depends only on $d$, $\kappa_1$, $\kappa_2$, 
and $\Omega$.
\end{theorem}

We remark that if $\Omega$ is $C^2$ and $u_\varep =0$ or $\frac{\partial u_\varep}{\partial \nu_\varep}=0$
on $\partial\Omega$, the  $O(\varep)$ estimate
\begin{equation}\label{L-2}
\| u_\varep -u_0\|_{L^2(\Omega)} \le C\, \varep \| F\|_{L^2(\Omega)}
\end{equation}
was proved in \cite{Suslina-2013, Suslina-2013-siam} 
for a broader class of elliptic operators with measurable periodic coefficients, which contains the systems of
elasticity considered here (also see \cite{Griso-2004,Onofrei-2007,KLS2,KLS3} and their references for related work
on convergence rates).
Note that $q=\frac{2d}{d-1}>2$ and $\| u_0\|_{H^2(\Omega)} \le C\, \| F\|_{L^2(\Omega)}$,
if $\Omega$ is $C^2$, $\mathcal{L}_0 (u_0)=F$ in $\Omega$
with  $u_0=0$ or $\frac{\partial u_0}{\partial\nu_0} =0$
on $\partial\Omega$.
Thus our estimate (\ref{sharp-rate}) is stronger than (\ref{L-2}).
In the case of scalar elliptic equations with Dirichlet condition $u_\varep=0$ on $\partial\Omega$,
 it is known that $\| u_\varep -u_0\|_{L^q(\Omega)} \le C \varep \| F\|_{L^p(\Omega)}$,
where $1<p<d$ and $\frac{1}{q}=\frac{1}{p} -\frac{1}{d}$  (see \cite[p.1234]{KLS3}).
However, Theorem \ref{main-theorem-L} seems to be the first result on the sharp $O(\varep)$
 estimate of $u_\varep -u_0$ in $L^q(\Omega)$ with $q>2$ for elliptic systems with bounded measurable 
 periodic coefficients.

As we indicated above, the proof of Theorems \ref{main-theorem-1} and \ref{main-theorem-2}
only uses the energy estimates in $L^2$ for $\mathcal{L}_\varep$ and thus requires no smoothness assumptions
on the coefficients.
In the second part of this paper we apply the similar ideas in the $L^p$ setting for $1<p<\infty$.
To do this we first establish the $W^{1,p}$ estimates for the systems 
\begin{equation}\label{div}
\mathcal{L}_\varep (u_\varep) =\text{div} (h) \quad \text{ in } \Omega,
\end{equation}
where $h=(h_i^\alpha)\in L^p(\Omega; \mathbb{R}^{d\times d})$,
with either the Dirichlet or Neumann boundary conditions, under the additional assumptions that
$\Omega$ is $C^1$ and $A=A(y)$ belongs to $V\!M\!O(\rd)$.
As a result, the $L^p$ analogues of estimates (\ref{main-estimate-1})
and (\ref{main-estimate-2}) are proved under these additional conditions,
which are more or less sharp.
Consequently, by combining the $L^p$ estimates on the boundary layer  $\Omega_\varep$
 with local estimates for $\mathcal{L}_1$, which hold for H\"older
continuous coefficients,
we may obtain the uniform Rellich estimates in $L^p$ for solutions of $\mathcal{L}_\varep (u_\varep)=0$
in $C^1$ domains under the assumptions that $A$ is H\"older continuous and satisfies (\ref{ellipticity})-(\ref{periodicity}).
By the method of layer potentials, this will lead to the uniform solvability of the
$L^p$ Dirichlet, regularity, and Neumann problems in $C^1$ domains
(details will be provided in a separate work).
Previously, these results in $L^p$ are known only in $C^{1, \alpha}$ domains
for operators $\mathcal{L}_\varep$ with H\"older continuous coefficients
satisfying (\ref{s-ellipticity}) and $A^*=A$   \cite{KLS1}.
We remark that the $W^{1,p}$ estimates (local or global) for operators with nonsmooth coefficients in
nonsmooth domains are of interest in their own rights and
have been studied extensively in recent years (see \cite{CP-1998, Wang-2003, Byun-Wang-2004,
Shen-2005, Byun-Wang-2005, Krylov-2007,  Byun-Wang-2008,
Shen-2008, Dong-2010, KLS1, Geng-2012, GSS} and their references).
Our approach to the $W^{1,p}$ estimates is based on a real-variable argument, which originated in 
\cite{CP-1998} and further developed in \cite{Wang-2003, Shen-2005, Shen-2007}.
The required (weak) reverse H\"older estimates at the boundary are proved by combining the interior 
Lipschitz estimates down to the scale $\varep$ with boundary $C^\alpha$ 
estimates.

In the third part of this paper we will study the boundary Lipschitz estimates,
uniform down to the scale $\varep$, for solutions in $C^{1, \alpha}$ domains
with the Dirichlet or Neumann conditions.
Let 
\begin{equation}\label{D}
\aligned
D_r  & =\Big\{ (x^\prime, x_d)\in \rd: |x^\prime|<r \text{ and } \psi(x^\prime)
<x_d < \psi(x^\prime) +r \Big\},\\
\Delta_r & =\Big\{ (x^\prime, x_d)\in \rd: |x^\prime|<r \text{ and } x_d=\psi(x^\prime) \Big\},
\endaligned
\end{equation}
where $\psi:\mathbb{R}^{d-1} \to \mathbb{R}$ is a $C^{1, \alpha}$ function
for some $\alpha>0$ with $\psi (0)=0$ and $\|\nabla \psi\|_{C^\alpha(\mathbb{R}^{d-1})}\le M$.

\begin{theorem}\label{main-theorem-3}
Suppose that $A$ satisfies conditions (\ref{ellipticity})-(\ref{periodicity}).
Let $u_\varep\in H^1(D_1;\rd)$ be a weak solution to
\begin{equation}\label{DP-2}
\mathcal{L}_\varep (u_\varep)=F  \quad \text{ in } D_1 \quad \text{ and } \quad u_\varep =f \quad \text{ on } \Delta_1.
\end{equation}
Then, for $\varep \le r <1$,
\begin{equation}\label{main-estimate-3}
\left(\average_{D_r} |\nabla u_\varep|^2\right)^{1/2}
\le C \left\{ \left(\average_{D_1} |\nabla  u_\varep|^2\right)^{1/2}
+\| f\|_{C^{1, \sigma}(\Delta_1)}
+\| F\|_{L^p(D_1)} \right\},
\end{equation}
where $p>d$ and $\sigma\in (0, \alpha)$.
The constant $C$ depends only on $d$, $\kappa_1$, $\kappa_2$,
$p$, $\sigma$, and $(\alpha, M)$.
\end{theorem}

\begin{theorem}\label{main-theorem-4}
Suppose that $A$ satisfies (\ref{ellipticity})-(\ref{periodicity}).
Let $u_\varep\in H^1(D_1;\rd)$ be a weak solution to
\begin{equation}\label{NP-2}
\mathcal{L}_\varep (u_\varep)=F  \quad  \text{ in } D_1\quad
\text{ and } \quad \frac{\partial u_\varep}{\partial \nu_\varep}
 =g\quad \text{ on } \Delta_1.
\end{equation}
Then, for $\varep \le r <1$,
\begin{equation}\label{main-estimate-4}
\left(\average_{D_r} |\nabla u_\varep|^2\right)^{1/2}
\le C \left\{ \left(\average_{D_1} | \nabla u_\varep|^2\right)^{1/2}
+\| g\|_{C^{\sigma}(\Delta_1)}
+\| F\|_{L^p(D_1)} \right\},
\end{equation}
where $p>d$ and $\sigma\in (0, \alpha)$.
The constant $C$ depends only on $d$, $\kappa_1$, $\kappa_2$,
$p$, $\sigma$, and $(\alpha, M)$.
\end{theorem}

As in the case of Rellich estimates, under additional smoothness conditions on $A$, 
using local Lipschitz estimates for $\mathcal{L}_1$ and a blow-up argument,
one may derive  from Theorems \ref{main-theorem-3} and
\ref{main-theorem-4} the full boundary Lipschitz estimates
\begin{equation}\label{Lip-estimate-1}
\|\nabla u_\varep\|_{L^\infty(D_{1/2})}
\le C \left\{ \left(\average_{D_1} | u_\varep|^2\right)^{1/2}
+\| f\|_{C^{1, \sigma}(\Delta_1)}
+\| F\|_{L^p(D_1)} \right\}
\end{equation}
for solutions of (\ref{DP-2}), and
\begin{equation}\label{Lip-estimate-2}
\|\nabla u_\varep\|_{L^\infty(D_{1/2})}
\le C \left\{ \left(\average_{D_1} | u_\varep|^2\right)^{1/2}
+\| g\|_{C^{ \sigma}(\Delta_1)}
+\| F\|_{L^p(D_1)} \right\}
\end{equation}
for solutions of (\ref{NP-2}).
We remark that for elliptic systems satisfying the ellipticity condition
(\ref{s-ellipticity}), the periodicity condition (\ref{periodicity}) and
the H\"older continuity condition,
the estimate (\ref{Lip-estimate-1}) was proved in \cite{AL-1987}, while
(\ref{Lip-estimate-2}) was established  in \cite{KLS1} under the additional symmetry condition 
$A^*=A$. This symmetry condition was removed recently in \cite{Armstrong-Shen-2014}

Our proof of Theorems \ref{main-theorem-3} and \ref{main-theorem-4}
also uses the function $w_\varep$, given by (\ref{w}). As a consequence of its estimates in 
$L^2$, for each $r \in (\varep, 1/4)$,
we are able to construct a function $v$ such that 
$\mathcal{L}_0 (v)=F$ in $ D_r$ with the same (Dirichlet or Neumann) data on $\Delta_r$ as $u_\varep$, and
$$
\left(\average_{D_r} |u_\varep -v|^2\right)^{1/2}
\le C \left(\frac{\varep}{r} \right)^{1/2}
\left\{ \left(\average_{D_{2r}} |u_\varep|^2\right)^{1/2}
+\text{terms involving given data} \right\}.
$$
This allows us to use a general scheme for establishing  Lipschitz estimates down to the scale $\varep$,
which was formulated recently in \cite{Armstrong-Smart-2014} and used for interior estimates
in stochastic homogenization with random coefficients 
(also see \cite{Armstrong-Mourrat-2015, Armstrong-Daniel-2015} as well as related work 
in \cite{Otto-2011, Otto-2012, Otto-2015, Otto-2014}).
 Our argument is similar to (and somewhat simpler than) that in \cite{Armstrong-Shen-2014},
 where the scheme was adapted to prove the full boundary Lipschitz estimates for second-order
 elliptic systems with almost-periodic 
 and H\"older continuous coefficients.
 As indicated earlier, we have been able to completely avoid the use of compactness methods
 (even in the case of $C^\alpha$ estimates).
 Although it is possible to prove the interior Lipschitz estimates as well as the boundary
 $C^\alpha$ estimates, down to the scale $\varep$ without smoothness,
 by the compactness methods, as demonstrated in \cite{AL-1987, Gu-Shen-2015},
 the compactness methods for boundary Lipschitz estimates require the same estimates for boundary correctors,
 which are not easy to establish \cite{AL-1987, KLS1}.
 
 The paper is organized as follows.
 In Section 2 we establish some key convergence results in $H^1$.
 These results are used in Section 3 to prove Theorems \ref{main-theorem-1} and \ref{main-theorem-2}.
 In Section 4 we study the convergence rates in $L^q$ for $q=\frac{2d}{d-1}$ and
 give the proof of Theorem \ref{main-theorem-L}, which uses the estimates in Theorems \ref{main-theorem-1}
 and \ref{main-theorem-2} as well as a duality argument.
 In Sections 5 and 6 we obtain the boundary $C^\alpha$ and $W^{1,p}$ estimates, 
 respectively, in $C^1$ domains for operators with VMO coefficients.
 These estimates are used in Section 7 to establish the $L^p$ analogues of
 (\ref{main-estimate-1}) and (\ref{main-estimate-2}) in $C^1$ domains.
  Finally, Theorem \ref{main-theorem-3} is proved in Section 8, and
  Section 9 contains the proof of Theorem \ref{main-theorem-4}.

Throughout the paper we use $\average_E u =\frac{1}{|E|}\int_E u$
 to denote the average of $u$ over the set $E$.
 We will use $C$ and $c$ to denote constants that may depend on $d$, $\kappa_1$,
 $\kappa_2$, $A$ and $\Omega$, but never on $\varep$.

\noindent{\bf Acknowledgement.}
The author thanks Carlos E. Kenig for several very helpful discussions regarding this work.
The author also would like to thank Scott N. Armstrong for insightful conversations and discussions regarding 
the work \cite{Armstrong-Smart-2014}.



\section{Convergence rates in $H^1$}
\setcounter{equation}{0}

In this section we establish certain results on convergence rates in $H^1$, which will play a crucial role
in the proof of our main results. Throughout the section we assume that
$A=A(y)$ satisfies (\ref{ellipticity}-(\ref{periodicity}) and
$\Omega$ is a bounded Lipschitz domain in $\rd$.

Let $\chi =\big( \chi_j^\beta (y)\big) =\big( \chi_j^{\alpha\beta} (y) \big)$ denote the matrix of correctors for 
$\mathcal{L}_\varep$, where  $1\le j, \alpha, \beta \le d$.
This means that $\chi_j^\beta \in H^1_{\loc}(\rd;\rd)$ is 1-periodic, $\int_Y \chi_j^\beta=0$, and
\begin{equation}\label{corrector-def}
\mathcal{L}_1 (\chi_j^\beta) =-\mathcal{L}_1 (P_j^\beta) \quad \text{ in } \rd,
\end{equation}
where $Y=[0,1)^d$ and $P_j^\beta =y_j (0, \dots, 1, \dots, 0)$ with $1$ in the $\beta^{th}$ position.
The homogenized operator is given by $\mathcal{L}_0=-\text{\rm div} \big(\widehat{A}\nabla \big)$, where
 $\widehat{A}= \big(\widehat{a}_{ij}^{\alpha\beta}\big)$ is the matrix of effective coefficients with
\begin{equation}\label{hat}
\widehat{a}_{ij}^{\alpha\beta}
 =\average_Y \Big\{ 
a_{ij}^{\alpha\beta}  + a_{ik}^{\alpha\gamma} \frac{\partial}{\partial y_k} \Big(\chi_j^{\gamma\beta}\Big)
\Big\}.
\end{equation}
It is known that the constant matrix $\widehat{A}$ satisfies the elasticity condition (\ref{ellipticity}) \cite{Oleinik-1992,Jikov-1994}.
Define
\begin{equation}\label{b}
b_{ij}^{\alpha\beta} (y)
 =a_{ij}^{\alpha\beta}  + a_{ik}^{\alpha\gamma} \frac{\partial}{\partial y_k} \Big(\chi_j^{\gamma\beta}\Big)
-\widehat{a}_{ij}^{\alpha\beta}.
\end{equation}
By the definition of $\widehat{A}$ and (\ref{corrector-def}),
\begin{equation}\label{c-1}
\int_{Y} b_{ij}^{\alpha\beta} =0 \quad \text{ and } \quad \frac{\partial}{\partial y_i} 
\Big( b_{ij}^{\alpha\beta} \Big) =0.
\end{equation}
It follows that there exist $\phi_{k i j}^{\alpha\beta}\in H^1_{\loc}(\rd)$ such that
$\phi_{k i j}^{\alpha\beta}$ is 1-periodic, 
\begin{equation}\label{phi}
b_{ij}^{\alpha\beta} =\frac{\partial}{\partial y_k} \Big(\phi_{k i j}^{\alpha\beta} \Big)
\quad \text{ and } \quad \phi_{k i j}^{\alpha\beta}
=-\phi_{ i k j}^{\alpha\beta}
\end{equation}
(see e.g. \cite{Jikov-1994, KLS2}).

Fix $\varphi \in C_0^\infty(B(0,1/4))$ such that $\varphi\ge 0$ and $\int_{\rd} \varphi=1$.
Define
\begin{equation}\label{K}
K_\varep (f)(x) =f *\varphi_\varep (x)=\int_{\rd} f(x-y) \, \varphi_\varep (y)\, dy,
\end{equation}
where $\varphi_\varep (y)=\varep^{-d} \varphi (y/\varep)$.

\begin{lemma}\label{lemma-2.00}
Let $f\in L^p(\rd)$ for some $1\le p<\infty$.
Then for any $g\in L^p_{\loc}(\rd)$,
\begin{equation}\label{2.00-1}
\| g(x/\varep) K_\varep (f)\|_{L^p(\rd)} \le C \sup_{x\in \rd} \left(\average_{B(x,1)} |g|^p\right)^{1/p}
\| f\|_{L^p(\rd)},
\end{equation}
where $C$ depends only on $d$.
\end{lemma}

\begin{proof}
By H\"older's inequality,
$$
|K_\varep (f) (x)|^p \le \frac{C}{|B(0, \varep)|} \int_{\rd} |f(y)|^p \chi_{B(x,\varep)} (y)\, dy,
$$
from which the estimate (\ref{2.00-1}) follows readily by Fubini's Theorem.
\end{proof}

It follows from (\ref{2.00-1}) that if $g\in L^p_{\loc} (\rd)$ and is 1-periodic, then
\begin{equation} \label{2.00-0}
\| g(x/\varep) K_\varep (f) \|_{L^p(\rd)}
\le C\, \| g\|_{L^p(Y)} \| f\|_{L^p(\rd)}.
\end{equation}

\begin{lemma}\label{lemma-p}
Let $f\in W^{1, q}(\rd)$ for some $1<q<\infty$. Then
\begin{equation}\label{L-p}
\| K_\varep (f) -f \|_{L^q(\rd)} \le C  \varep \, \| \nabla f\|_{L^q (\rd)}.
\end{equation}
Moreover, if  $p=\frac{2d}{d+1}$,
\begin{equation}\label{p}
\aligned
\| K_\varep (f)\|_{L^2(\rd)}
 & \le C \varep^{-1/2} \| f\|_{L^p(\rd)},\\
 \| f -K_\varep (f)\|_{L^2(\rd)}
 & \le C\, \varep^{1/2} \| \nabla f\|_{L^p(\rd)}.
 \endaligned
\end{equation}
The constant $C$ depends only on $d$.
\end{lemma}

\begin{proof}
To see (\ref{L-p}), we note that
$$
\| f(\cdot-y) -  f (\cdot) \|_{L^q(\rd)} \le  |y|\,  \| \nabla f\|_{L^q(\rd)}
$$
for any $y\in \rd$.
Thus, by Minkowski's inequality,
$$
\aligned
\| K_\varep (f) -f \|_{L^q(\rd)} & \le \int_{\rd} \varphi_\varep (y) 
\| f(\cdot-y) -  f (\cdot) \|_{L^q(\rd)} dy \\
& \le  \int_{\rd} \varphi_\varep (y) |y|\, dy \, \|\nabla  f\|_{L^q(\rd)}\\
&= C\varep\, \|\nabla f\|_{L^q(\rd)}.
\endaligned
$$

Next, by Parseval's Theorem and H\"older's inequality,
$$
\aligned
\int_{\rd} |K_\varep (f)|^2\, dx
&= \int_{\rd} |\hat{\varphi} (\varep \xi)|^2 |\hat{f} (\xi)|^2\, d\xi\\
&\le \left(\int_{\rd} |\hat{\varphi}(\varep \xi)|^{2d}\, d\xi\right)^{1/d}
\| \hat{f}\|_{L^{p^\prime}(\rd)}^2\\
& \le  C\, \varep^{-1} \| f\|^2_{L^p(\rd)},
\endaligned
$$
where $\hat{f}$ denotes the Fourier transform of $f$, and
we have used the  Hausdorff-Young inequality $\| \hat{f}\|_{L^{p^\prime}(\rd)} \le  \| f\|_{L^p(\rd)}$.
This gives the first inequality in (\ref{p}).
To see the second inequality, we note that $\hat{\varphi} (0)=\int_{\rd} \varphi =1$.
It follows that
$$
\aligned
\| f-K_\varep (f)\|_{L^2(\rd)}
&\le C \left\{\int_{\rd} |\hat{\varphi} (\varep \xi) -\hat{\varphi} (0)|^{2d} |\xi|^{-2d}\, d\xi\right\}^{1/(2d)}
\|\widehat{\nabla f } \|_{L^{p^\prime} (\rd)}\\
&\le C \varep^{1/2} \|\nabla f\|_{L^p(\rd)},
\endaligned
$$
where we have used $|\hat{\varphi} (\xi) -\hat{\varphi} (0)|\le C |\xi|$ for the last step.
\end{proof}

\begin{lemma}\label{main-lemma}
Let $u_\varep, u_0\in H^1(\Omega; \rd)$.
Suppose that $\mathcal{L}_\varep (u_\varep )=\mathcal{L}_0 (u_0)$ in $\Omega$
and either $u_\varep = u_0$ or $\frac{\partial u_\varep}{\partial\nu_\varep}
=\frac{\partial u_0}{\partial \nu_0}$ on $\partial\Omega$.
Let
$$
w^\alpha_\varep= u_\varep^\alpha -{u}^\alpha_0 -\varep \chi^{\alpha\beta}_j (x/\varep) K^2_\varep \Big(\frac{\partial 
{u}^\beta_0}{\partial x_j} \eta_\varep \Big),
$$
where
$\eta_\varep \in C_0^\infty(\Omega)$ and supp$(\eta_\varep)
\subset \big\{ x\in \Omega: \, \text{dist}(x, \partial\Omega) \ge 3 \varep \big\}$.
Then
\begin{equation}\label{f-0}
\aligned
\int_\Omega A(x/\varep)\nabla w_\varep \cdot \nabla w_\varep\, dx
=& \int_\Omega \Big[ \widehat{A} - A(x/\varep) \Big]
\Big[ \nabla u_0 - K_\varep^2 \big( (\nabla u_0)\eta_\varep \big) \Big] \cdot \nabla w_\varep\, dx\\
& \qquad -\int_\Omega B(x/\varep) K_\varep^2 \big((\nabla u_0)\eta_\varep \big) \cdot \nabla w_\varep\, dx\\
&\qquad -\varep \int_\Omega A(x/\varep) \chi(x/\varep) \nabla K_\varep^2 \big( (\nabla u_0) \eta_\varep \big)\cdot 
\nabla w_\varep\, dx,
\endaligned
\end{equation}
where $B(y)=(b_{ij}^{\alpha\beta}(y))$ is defined in (\ref{b}).
\end{lemma}

\begin{proof}
We first note that if $u_\varep =u_0$ on $\partial\Omega$, then $w_\varep \in H^1_0(\Omega; \rd)$,
as $K_\varep^2 \big( (\nabla u_0) \eta_\varep)\in C_0^\infty(\Omega)$.
Since $\mathcal{L}_\varep (u_\varep)=\mathcal{L}_0 (u_0)$ in $\Omega$,
it follows that
\begin{equation}\label{f-01}
\int_\Omega A(x/\varep)\nabla u_\varep \cdot \nabla w_\varep\, dx
=\int_\Omega \widehat{A} \nabla u_0 \cdot \nabla w_\varep\, dx.
\end{equation}
In the case of the Neumann condition $\frac{\partial u_\varep}{\partial\varep}
=\frac{\partial u_0}{\partial \nu_0}$ on $\partial \Omega$, the equation (\ref{f-01})
continues to hold. This is because $w_\varep\in H^1(\Omega; \rd)$ and both sides  of (\ref{f-01}) equal to
$$
\langle \mathcal{L}_0 (u_0), w_\varep\rangle_{\big(H^1(\Omega)\big)^\prime \times H^1 (\Omega)}
+\big\langle \frac{\partial u_0}{\partial \nu_0}, w_\varep \big\rangle_{H^{-1/2}(\partial\Omega)\times 
H^{1/2}(\partial\Omega)}.
$$
Using (\ref{f-01}), we obtain 
$$
\aligned
\int_\Omega A(x/\varep)\nabla w_\varep \cdot \nabla w_\varep\, dx
&=\int_\Omega \Big[ \widehat{A} -A(x/\varep) \Big] \nabla u_0 \cdot \nabla w_\varep\, dx\\
&\qquad -\int_\Omega A(x/\varep) \nabla \chi (x/\varep) K_\varep^2 \big( (\nabla u_0)\eta_\varep\big) 
\cdot \nabla w_\varep\, dx\\
& \qquad -\varep \int_\Omega A(x/\varep) \chi(x/\varep) \nabla K_\varep^2 \big( (\nabla u_0) \eta_\varep \big)\cdot 
\nabla w_\varep\, dx,
\endaligned
$$
from which the formal (\ref{f-0}) follows by the definition of $B(y)$.
\end{proof}

\begin{lemma}\label{main-lemma-1}
Let $\phi  (y)=\big( \phi_{k i j}^{\alpha\beta} (y) \big)$ be defined by (\ref{phi}).
Then
\begin{equation}\label{f-02}
\int_\Omega B(x/\varep) K_\varep^2 \big((\nabla u_0)\eta_\varep\big) \cdot \nabla w_\varep\, dx
=-\varep \int_\Omega
\phi_{k i j }^{\alpha\beta} (x/\varep) \frac{\partial w_\varep^\alpha}{\partial x_i}\cdot 
\frac{\partial}{\partial x_k} K_\varep^2 \Big( \frac{\partial u_0^\beta}{\partial x_j} \eta_\varep\Big)\, dx.
\end{equation}
\end{lemma}

\begin{proof}
Using (\ref{phi}), we see that
$$
\aligned
 B(x/\varep) K_\varep^2 \big((\nabla u_0)\eta_\varep\big) \cdot \nabla w_\varep
&=b_{ij}^{\alpha\beta}(x/\varep) K_\varep^2 \Big(\frac{\partial u_0^\beta}{\partial x_j} \eta_\varep\Big)
\cdot \frac{\partial w_\varep^\alpha}{\partial x_i} \\
&=\varep \frac{\partial}{\partial x_k} \Big( \phi_{kij}^{\alpha\beta} (x/\varep)\Big)
K_\varep^2 \Big(\frac{\partial u_0^\beta}{\partial x_j} \eta_\varep\Big)
\cdot \frac{\partial w_\varep^\alpha}{\partial x_i} \\
&=\varep \frac{\partial}{\partial x_k} 
\left\{ \phi_{kij}^{\alpha\beta} (x/\varep) \frac{\partial w_\varep^\alpha}{\partial x_i} \right\}
K_\varep^2 \Big(\frac{\partial u_0^\beta}{\partial x_j} \eta_\varep\Big),
\endaligned
$$
from which the equation (\ref{f-02}) follows readily.
\end{proof}

\begin{lemma}\label{main-lemma-2}
Let $u_\varep$ $(\varep\ge 0)$ be a solution to the Dirichlet problem (\ref{DP-1}) or the Neumann problem 
(\ref{NP-1}).
Let $w_\varep$ be defined as in Lemma \ref{main-lemma} with $\eta_\varep$ satisfying
\begin{equation}\label{eta}
\left\{ \aligned
&\eta_\varep \in C_0^\infty (\Omega), \ 0\le \eta\le 1,\\
&\text{supp} (\eta_\varep) \subset \big\{ x\in \Omega: \text{dist}(x, \partial \Omega)\ge 3\varep \big\},\\
&\eta_\varep =1 \text{ on } \big\{ x\in \Omega: \text{dist}(x,\partial\Omega)\ge 4\varep\big\} ,\\
& |\nabla \eta_\varep |\le C \varep^{-1}.
\endaligned
\right.
\end{equation}
Then
\begin{equation}\label{f-03}
\aligned
&\Big| \int_\Omega A(x/\varep)\nabla w_\varep \cdot \nabla w_\varep\, dx \Big|\\
& \le C\,  \|\nabla w_\varep\|_{L^2(\Omega)}
\Big\{ \|\nabla u_0\|_{L^2(\Omega_{4\varep})}
+\| (\nabla u_0)\eta_\varep -K_\varep \big((\nabla u_0)\eta_\varep\big)\|_{L^2(\Omega)}\\
&\qquad\qquad\qquad\qquad\qquad\qquad\qquad
+\varep \| K_\varep \big( (\nabla^2 u_0)\eta_\varep\big) \|_{L^2(\Omega)}\Big\}.
\endaligned
\end{equation}
\end{lemma}

\begin{proof}
It follows from Lemmas \ref{main-lemma} and \ref{main-lemma-1} by the Cauchy inequality that
$$
\aligned
& \Big| \int_\Omega A(x/\varep)\nabla w_\varep \cdot \nabla w_\varep\, dx \Big|\\
&\le C\, \|\nabla w_\varep\|_{L^2(\Omega)}
\Big\{ \| \nabla u_0 - K_\varep^2 \big((\nabla u_0)\eta_\varep\big)\|_{L^2(\Omega)}
+\varep \| \chi(x/\varep) \nabla K_\varep^2 \big( (\nabla u_0)\eta_\varep\big)\|_{L^2(\Omega)}\\
& \qquad\qquad\qquad\qquad\qquad\qquad\qquad\qquad\qquad
+\varep \| \phi (x/\varep) \nabla K_\varep^2 \big( (\nabla u_0)\eta_\varep\big)\|_{L^2(\Omega)}\Big\}\\
&\le  C\, \|\nabla w_\varep\|_{L^2(\Omega)}
\Big\{ \| \nabla u_0 - K_\varep^2 \big((\nabla u_0)\eta_\varep\big)\|_{L^2(\Omega)}
+\varep \| \nabla K_\varep \big( (\nabla u_0)\eta_\varep\big)\|_{L^2(\Omega)}\Big\},
\endaligned
$$
where we have used Lemma \ref{lemma-2.00} as well as the fact that $\chi, \phi \in L^2_{\loc}(\rd)$
and are 1-periodic for the last inequality.
Finally, we observe that
$$
\aligned
\| \nabla u_0 - K_\varep^2 \big((\nabla u_0)\eta_\varep\big)\|_{L^2(\Omega)}
 &\le \|(\nabla u_0) (1-\eta_\varep)\|_{L^2(\Omega)}
+\| (\nabla u_0)\eta_\varep - K_\varep \big( (\nabla u_0)\eta_\varep\big) \|_{L^2(\Omega)}\\
&\qquad\qquad \qquad
+\| K_\varep \Big( \big(\nabla u_0)\eta_\varep -K_\varep \big((\nabla u_0)\eta_\varep\big) \Big)\|_{L^2(\Omega)}\\
\\
&\le \|\nabla u_0\|_{L^2(\Omega_{4\varep})}
+C\, \| (\nabla u_0)\eta_\varep - K_\varep \big( (\nabla u_0)\eta_\varep\big) \|_{L^2(\Omega)}.
\endaligned
$$
This completes the proof.
\end{proof}

Finally, we are in a position to state and prove the main results of this section.

\begin{theorem}\label{theorem-2.1}
Suppose that $A(y)$satisfies (\ref{ellipticity})-(\ref{periodicity}).
Let $\Omega$ be a bounded Lipschitz domain.
Let $u_\varep$ $(\varep\ge 0)$ be the solutions to the Dirichlet problem (\ref{DP-1})
in $\Omega$ with $f\in H^1(\partial\Omega; \rd)$ and $F\in L^p(\Omega; \rd)$, where
$p=\frac{2d}{d+1}$. Then
\begin{equation}\label{main-estimate-2.1}
\|  u_\varep -u_0 - \varep \chi_j^\beta (x/\varep) K^2_\varep \Big(\frac{\partial {u}_0^\beta}{\partial x_j} \eta_\varep\Big)\|_{H^1_0(\Omega)}
\le C \varep^{1/2} 
\Big\{ \| f\|_{H^1(\partial\Omega)} +\| F\|_{L^p(\Omega)} \Big\},
\end{equation}
where $\eta_\varep\in C_0^\infty(\Omega)$ satisfies (\ref{eta}).
The constant $C$ depends only on $d$, $\kappa_1$, $\kappa_2$, and the Lipschitz character of
$\Omega$.
\end{theorem}

\begin{proof}
Let $w_\varep$ denote the function in the l.h.s. of (\ref{main-estimate-2.1}).
Since $w_\varep \in H_0^1(\Omega; \rd)$, it follows from (\ref{f-03}) 
by the first Korn inequality that
\begin{equation}\label{2.1-0}
\aligned
\| w_\varep\|_{H^1_0(\Omega)}
\le C 
\Big\{ \|\nabla u_0\|_{L^2(\Omega_{4\varep})}
& +\| (\nabla u_0)\eta_\varep -K_\varep \big((\nabla u_0)\eta_\varep\big)\|_{L^2(\Omega)}\\
&+\varep \| K_\varep \big( (\nabla^2 u_0)\eta_\varep\big) \|_{L^2(\Omega)}\Big\}.
\endaligned
\end{equation}
To bound the r.h.s. of (\ref{2.1-0}),
we write $u_0= v+h$, where 
$$
v(x)=\int_\Omega \Gamma_0 (x-y) F(y)\, dy
$$
and $\Gamma_0 (x)$ denotes the matrix of fundamental solutions for the homogenized operator $\mathcal{L}_0$ in 
$\rd$, with pole at the origin. Note that $\mathcal{L}_0 (v)=F$ in $\Omega$, and by the well known
singular integral and fractional integral estimates,
\begin{equation}\label{2.1-2}
\| \nabla^2 v\|_{L^p(\rd)}  +\| \nabla v\|_{L^{p^\prime}(\rd)}\le C_p\, \| F\|_{L^p(\Omega)},
\end{equation}
where we have used the observation $\frac{1}{p^\prime} =\frac{1}{p}-\frac{1}{d}$.
Let $\mathbf{e}=(e_1, \dots, e_d)\in C_0^1(\rd;\rd)$ be a vector field such that 
$\langle \mathbf{e}, n\rangle\ge c_0>0$ on $\partial\Omega$ and $|\nabla \mathbf{e}|\le C r_0^{-1}$,
where $r_0=\text{diam}(\Omega)$ and $n$ denotes the outward unit normal to $\partial\Omega$.
It follows from the divergence theorem that
\begin{equation}\label{2.1-3}
\aligned
c_0\int_{\partial\Omega} |\nabla v|^2\, d\sigma 
&\le \int_{\partial\Omega} |\nabla v|^2 \langle \mathbf{e}, n\rangle\, d\sigma \\
&=\int_\Omega |\nabla v|^2\,  \text{div} (\mathbf{e})\, dx 
+\int_\Omega e_i \frac{\partial}{\partial  x_i} \nabla v \cdot \nabla v\, dx\\
& \le C \left\{ r_0^{-1} \int_\Omega |\nabla v|^2\, dx
+\int_\Omega |\nabla v|\,  |\nabla^2 v|\, dx \right\}\\
& \le C \Big\{ r_0^{-1} \|\nabla v\|_{L^2(\Omega)}^2 +\| \nabla v\|_{L^{p^\prime}(\Omega)}
\|\nabla^2 v\|_{L^{p}(\Omega)} \Big\}\\
&\le C\, \| F\|^2_{L^p(\Omega)},
\endaligned
\end{equation}
where we have used (\ref{2.1-2}) for the last step.
Note that the same argument also gives
$\| \nabla v\|_{L^2(S_t)} \le C\, \|F\|_{L^p(\Omega)}$, where
$S_t =\big\{ x\in \rd: \text{dist}(x, \partial\Omega)=t \big\}$ for $0<t<cr_0$.
Consequently, by the co-area formula, we obtain 
\begin{equation}\label{2.1-4}
\left\{ \frac{1}{r} \int_{\widetilde{\Omega}_r} |\nabla v|^2 \, dx \right\}^{1/2}
\le C\, \| F\| _{L^p(\Omega)},
\end{equation}
where $0<r< \text{diam}(\Omega)$ and $\widetilde{\Omega}_r =\{ x\in \rd: \text{dist} (x, \partial\Omega)<r \}$.

Next, we observe that $\mathcal{L}_0 (h)=0$ in $\Omega$ and
$$
\aligned
\| h\|_{H^1(\partial\Omega)}
&\le \| f\|_{H^1(\partial\Omega)} +\| v\|_{H^1(\partial\Omega)}\\
&\le \| f\|_{H^1(\partial\Omega)} +C\, \| F\|_{L^p(\Omega)},
\endaligned
$$
where we have used (\ref{2.1-3}) for the last inequality.
It follows from the estimates for solutions of the $L^2$ regularity problem in Lipschitz domains
for the operator $\mathcal{L}_0$ in \cite{Dahlberg-1988, Verchota-1986} that
\begin{equation}\label{2.1-5}
\|(\nabla h)^*\|_{L^2(\partial\Omega)}
\le C\Big\{ 
\| f\|_{H^1(\partial\Omega)} +\| F\|_{L^p(\Omega)} \Big\},
\end{equation}
where $(\nabla h)^*$ denotes the nontangential maximal function of $\nabla h$.
This, together with (\ref{2.1-4}), gives
\begin{equation}\label{2.1-6}
\| \nabla u_0\|_{L^2(\Omega_r)}
\le C\, r^{1/2} 
\Big\{ 
\| f\|_{H^1(\partial\Omega)} +\| F\|_{L^p(\Omega)} \Big\}.
\end{equation}
for any $0<r<\text{diam}(\Omega)$. 
As a result, the first term in the r.h.s. of (\ref{2.1-0}) is bounded by
$C\varep^{1/2} \big\{ \| f\|_{H^1(\partial\Omega)} +\| F\|_{L^p(\Omega)} \big\}$.

To handle the third term in the r.h.s. of (\ref{2.1-0}),
we use Lemma \ref{lemma-p} to obtain 
\begin{equation}\label{2.1-8}
\aligned
\varep \| K_\varep\big((\nabla^2 u_0)\eta_\varep \big)\|_{L^2(\Omega)}
&\le \varep \| K_\varep\big((\nabla^2 v)\eta_\varep \big)\|_{L^2(\Omega)}
+\varep \| K_\varep\big((\nabla^2 h)\eta_\varep \big)\|_{L^2(\Omega)}\\
&\le C \varep^{1/2} \| (\nabla^2 v) \eta_\varep\|_{L^p(\Omega)}
+ C\varep \| (\nabla^2 h) \eta_\varep \|_{L^2(\Omega)} \\
&\le C \varep^{1/2} \| F\|_{L^p(\Omega)}
+C\varep \|\nabla^2 h\|_{L^2(\Omega\setminus \Omega_{3\varep})}.
\endaligned
\end{equation}
Since $\mathcal{L}_0 (\nabla h)=0$ in $\Omega$, we may use the interior estimate for $\mathcal{L}_0$,
$$
|\nabla^2 h(x)|\le \frac{C}{\delta (x) } \left(\average_{B(x, \delta(x)/8)} |\nabla h|^2\right)^{1/2},
$$
where $\delta (x)=\text{dist}(x, \partial\Omega)$, to show that
\begin{equation}\label{square-estimate}
\aligned
\|\nabla^2 h\|_{L^2(\Omega\setminus \Omega_{3\varep})}
 &\le C \|(\nabla h) [\delta(x)]^{-1} \|_{L^2(\Omega\setminus \Omega_\varep)}\\
&\le C \varep^{-1/2} \Big\{ \| f\|_{H^1(\partial\Omega)} +\| F\|_{L^p(\Omega)} \Big\},
\endaligned
\end{equation}
where the last inequality follows from (\ref{2.1-5}).
This, together with (\ref{2.1-8}), gives
\begin{equation}\label{2.1-8-1}
\varep \| K_\varep\big((\nabla^2 u_0)\eta_\varep \big)\|_{L^2(\Omega)}
\le C\varep^{1/2} \Big\{ \| f\|_{H^1(\partial\Omega)} +\| F\|_{L^p(\Omega)} \Big\}.
\end{equation}

Finally, to bound the second term in  the r.h.s. of (\ref{2.1-0}), we again write $u_0=v +h$ as before.
Note that by Lemma \ref{lemma-p},
$$
\aligned
\| (\nabla v)\eta_\varep -K_\varep \big((\nabla v)\eta_\varep\big)\|_{L^2(\Omega)}
&\le \|\nabla v -K_\varep (\nabla v)\|_{L^2(\rd)}
+\|(\nabla v)(1-\eta_\varep)\|_{L^2(\Omega)}\\
&\qquad\qquad\qquad
+\| K_\varep \big( (\nabla v)(1-\eta_\varep)\big)\|_{L^2(\Omega)}\\
&\le C\varep^{1/2}  \|\nabla^2 v\|_{L^p(\rd)}
+ C\, \| \nabla v\|_{L^2(\widetilde{\Omega}_{8\varep})}\\
&\le C \varep^{1/2} \| F\|_{L^p(\Omega)},
\endaligned
$$
where we have used (\ref{2.1-2}) and (\ref{2.1-4}) for the last inequality.
Also, by Lemma \ref{lemma-2.00},
$$
\aligned
\| (\nabla h) \eta_\varep -K_\varep \big( (\nabla h)\eta_\varep\big) \|_{L^2(\Omega)}
&\le C\varep \|\nabla \big( (\nabla h) \eta_\varep) \|_{L^2(\Omega)}\\
&\le C \Big\{ \varep \|\nabla^2 h\|_{L^2(\Omega \setminus \Omega_{3\varep})}
+\|\nabla h\|_{L^2 (\Omega_{4\varep})} \Big\}\\
&\le C \varep^{1/2}
\Big\{ \| f\|_{H^1(\partial\Omega)} +\| F\|_{L^p(\Omega)} \Big\}.
\endaligned
$$
Consequently, the second term in the r.h.s. of (\ref{2.1-0}) is dominated by
the r.h.s. of (\ref{main-estimate-2.1}).
This completes the proof of Theorem \ref{theorem-2.1}.
\end{proof}

The next theorem is an analogue of Theorem \ref{theorem-2.1} for the Neumann boundary conditions.

\begin{theorem}\label{theorem-2.2}
Suppose that $A=A(y)$ satisfies (\ref{ellipticity})-(\ref{periodicity}).
Let $\Omega$ be a bounded Lipschitz domain.
Let $u_\varep$ $(\varep\ge 0)$ be the solutions to the Neumann problem (\ref{NP-1})
in $\Omega$ with $g\in L^2(\partial\Omega; \rd)$ and $F\in L^p(\Omega; \rd)$, where
$p=\frac{2d}{d+1}$. Also assume that $ u_\varep, u_0 \perp \mathcal{R}$. Then
\begin{equation}\label{main-estimate-2.2}
\|  u_\varep -u_0 - \varep \chi_j^\beta (x/\varep) K^2_\varep \Big(\frac{\partial {u}_0^\beta}{\partial x_j} \eta_\varep\Big)\|_{H^1(\Omega)}
\le C \varep^{1/2} 
\Big\{ \| g\|_{L^2(\partial\Omega)} +\| F\|_{L^p(\Omega)} \Big\},
\end{equation}
where $\eta_\varep\in C_0^\infty(\Omega)$ satisfies (\ref{eta}).
The constant $C$ depends only on $d$, $\kappa_1$, $\kappa_2$, and the Lipschitz character of
$\Omega$.
\end{theorem}

\begin{proof}
The proof, which uses the estimate in Lemma \ref{main-lemma-2},
 is similar to that of Theorem \ref{theorem-2.1}.
We will only point out the differences and leave the details to the reader.

Let $w_\varep$ denote the function in the left hand side of (\ref{main-estimate-2.2}).
Let 
$$
\Big\{ \phi_j:  j=1, \dots, J=\frac{d(d+1)}{2}\Big\}
$$
 be an orthonormal basis of $\mathcal{R}$, as a subspace of 
$L^2(\Omega; \rd)$. By the second Korn inequality,
\begin{equation}\label{2.2-1}
\| w_\varep \|_{H^1(\Omega)}
 \le C \, \Big| \int_\Omega A(x/\varep)\nabla w_\varep\cdot \nabla w_\varep \, dx \Big|
+C \sum_{j=1}^{J} \Big| \int_\Omega w_\varep \cdot \phi_j \, dx \Big|.
\end{equation}
Since $u_\varep, u_0 \perp \mathcal{R}$, it follows that
$$
\aligned
\Big| \int_\Omega w_\varep \cdot \phi_j \, dx \Big|
 &\le C \varep\, \|\chi(x/\varep) K_\varep^2 \big( (\nabla u_0)\eta_\varep \big) \|_{L^2(\Omega)}\\
 &\le C \varep \, \| \nabla u_0 \|_{L^2(\Omega)}.
 \endaligned
 $$
This, together with (\ref{2.2-1}) and Lemma \ref{main-lemma-2}, shows that
\begin{equation}\label{2.2-3}
\aligned
\| w_\varep\|_{H^1(\Omega)}
 \le C 
\Big\{ \|\nabla u_0\|_{L^2(\Omega_{4\varep})} +\varep \|\nabla u_0\|_{L^2(\Omega)} 
& +\| (\nabla u_0)\eta_\varep -K_\varep \big((\nabla u_0)\eta_\varep\big)\|_{L^2(\Omega)}\\
&+\varep \| K_\varep \big( (\nabla^2 u_0)\eta_\varep\big) \|_{L^2(\Omega)}\Big\}.
\endaligned
\end{equation}

To bound the r.h.s. of (\ref{2.2-3}), we write $u_0 =v +h$, where
$v$ is the same as in the proof of Theorem \ref{theorem-2.1}.
Since $\mathcal{L}_0 (h)=0$ in $\Omega$ and
$$
\aligned
\| \frac{\partial h}{\partial \nu_0} \|_{L^2(\partial\Omega)}
& \le \| \frac{\partial u_0}{\partial \nu_0} \|_{L^2(\partial\Omega)}
+\| \frac{\partial v}{\partial \nu_0} \|_{L^2(\partial\Omega)}\\
&\le C \Big\{ \| g\|_{L^2(\partial\Omega)} +\| F\|_{L^p(\Omega)} \Big\},
\endaligned
$$
we may use the estimates  in \cite {Dahlberg-1988, Verchota-1986} 
for solutions of the $L^2$ Neumann problem for $\mathcal{L}_0$ in Lipschitz
domains to obtain
\begin{equation}\label{2.2-5}
\aligned
\|(\nabla h)^*\|_{L^2(\partial\Omega)}
&\le C \Big\{ \| g\|_{L^2(\partial \Omega)} +\| F\|_{L^p(\Omega)}
+ \sum_{j=1}^J \Big| \int_\Omega h \cdot \phi_j \Big| \Big\}\\
& \le C \Big\{ \| g\|_{L^2(\partial \Omega)} +\| F\|_{L^p(\Omega)} \Big\},
\endaligned
\end{equation}
where we have used the assumption $u_0\perp \mathcal{R}$.
With the nontangential maximal function estimate (\ref{2.2-5}) at our disposal,
the rest of the proof is  exactly the same as that of Theorem \ref{theorem-2.1}.
\end{proof}

\begin{remark}\label{remark-2.1}
{\rm
Since
$$
\|\chi(x/\varep) K_\varep^2 \big( (\nabla u_\varep) \eta_\varep\big) \|_{L^2(\Omega)}
\le C\, \|\nabla u_\varep\|_{L^2(\Omega)},
$$
it follows from the  estimate (\ref{main-estimate-2.1}) that
\begin{equation}\label{L-2-D}
\| u_\varep -u_0\|_{L^2(\Omega)}
\le C \, \varep^{1/2} 
\Big\{ \| f\|_{H^1(\partial\Omega)} + \| F\|_{L^2(\Omega)} \Big\},
\end{equation} 
where $\mathcal{L}_\varep (u_\varep)=\mathcal{L}_0(u_0)=F$ in $\Omega$ and
$u_\varep=u_0=f$ on $\partial\Omega$.
Similarly, the estimate (\ref{main-estimate-2.2}) implies that
\begin{equation}\label{L-2-N}
\| u_\varep -u_0\|_{L^2(\Omega)}
\le C \, \varep^{1/2} 
\Big\{ \| g\|_{L^2(\partial\Omega)} + \| F\|_{L^2(\Omega)} \Big\},
\end{equation} 
where $u_\varep, u_0$ are given in Theorem \ref{theorem-2.2}.
These $O(\varep^{1/2})$ estimate in $L^2$
are not sharp (see Section 4); but they will be sufficient  for us to establish the boundary $C^\alpha$ and Lipschitz estimates.
}
\end{remark}



\section{Proof of Theorems \ref{main-theorem-1} and \ref{main-theorem-2}}
\setcounter{equation}{0}

Theorems \ref{main-theorem-1} and \ref{main-theorem-2} are consequences of Theorems \ref{theorem-2.1}
and \ref{theorem-2.2}, respectively.
We give the proof of Theorem \ref{main-theorem-1}. Theorem \ref{main-theorem-2} follows from Theorem 
\ref{theorem-2.2} in the same manner.

Without loss of generality we may assume that
$$
\| f\|_{H^1(\partial\Omega)} + \| F\|_{L^p(\Omega)} =1.
$$
Let $w_\varep$ denote the function in the left hand side of (\ref{main-estimate-2.1}).
By Theorem \ref{theorem-2.1}, for $\varep\le r< \text{diam}(\Omega)$,
$$
\aligned
\|\nabla u_\varep\|_{L^2(\Omega_r)}
&\le \|\nabla u_0\|_{L^2(\Omega_r)}
+\|\nabla w_\varep\|_{L^2(\Omega)}
+\varep \|\nabla \big\{ \chi(x/\varep) K_\varep^2 \big( (\nabla u_0)\eta_\varep\big) \|_{L^2(\Omega_r)} \\
&\le C\, r^{1/2} 
+\|\nabla \chi(x/\varep) K_\varep^2 \big( (\nabla u_0)\eta_\varep\big)\|_{L^2(\Omega_r)}
+\varep\,  \| \chi(x/\varep) \nabla K_\varep^2 \big( (\nabla u_0)\eta_\varep\big) \|_{L^2(\Omega_r)}\\
&\le C\, r^{1/2} 
+ C\, \| K_\varep \big( (\nabla u_0)\eta_\varep\big)\|_{L^2(\Omega_{2r})}
+C \varep\,  \|  \nabla K_\varep \big( (\nabla u_0)\eta_\varep\big) \|_{L^2(\Omega_{2r})},
\endaligned
$$
where we have used (\ref{2.1-6}) and Lemma \ref{lemma-2.00} as well as the fact that the
operator $K_\varep$ is a convolution with a kernel supported in $B(0, \varep/4)$.
Note that by (\ref{2.1-6})  and (\ref{2.1-8-1}),

$$
\| K_\varep \big( (\nabla u_0)\eta_\varep\big)\|_{L^2(\Omega_{2r})}
\le C \, \| \nabla u_0\|_{L^2(\Omega_{3r})} \le C \, r^{1/2},
$$
and 
$$
\aligned
\varep \, \| \nabla K_\varep \big ((\nabla u_0)\eta_\varep\big) \|_{L^2(\Omega_{2r})}
&\le \varep\,  \|K_\varep \big((\nabla^2 u_0)\eta_\varep \big) \|_{L^2(\Omega_{2r})}
+ \varep \,  \| K_\varep \big((\nabla u_0) (\nabla \eta_\varep) \big) \|_{L^2(\Omega_{2r})} \\
& \le \varep\,  \|K_\varep \big((\nabla^2 u_0)\eta_\varep \big) \|_{L^2(\Omega_{2r})}
+C\, \|\nabla u_0\|_{L^2(\Omega_{3r})}\\
&\le C\, r^{1/2} .
\endaligned
$$
The proof of Theorem \ref{main-theorem-1} is complete.

\begin{remark}\label{remark-3.1}
{\rm
Under certain smoothness condition on $A$, it is possible to extend the Rellich estimates in \cite{Dahlberg-1988}
for the Lam\'e systems with constant coefficients to the operator $\mathcal{L}_1$ with variable coefficients
satisfying the condition (\ref{ellipticity}). 
We refer the reader to \cite{Kenig-Shen-2}, where this is done for the case  that the coefficients 
satisfy the ellipticity condition (\ref{s-ellipticity}).
It follows that if $\mathcal{L}_1 (u)=0$ in $D_2$, 
where $D_r$ is defined by (\ref{D}) with $\psi (0)=0$ and
$\|\nabla \psi\|_\infty \le M$, then
\begin{equation}\label{re-2.2}
\left\{
\aligned
\int_{\partial D_r} |\nabla u|^2 \, d\sigma 
& \le C \int_{\partial D_r} \Big|\frac{\partial u}{\partial \nu} \Big|^2\, d\sigma
+ \int_{D_r} |\nabla u|^2\, dx,\\
\int_{\partial D_r} |\nabla u|^2 \, d\sigma
& \le C \int_{\partial D_r} |\nabla_{\tan}  u|^2\, d\sigma  +\int_{D_r} |\nabla u|^2\, dx,
\endaligned
\right.
\end{equation}
for any $r\in (1, 3/2)$, where $C$ depends only on  $d$, $A$, and $M $.
By integrating both sides of the inequalities in (\ref{re-2.2}) with respect to $r$ over $(1,3/2)$,
we obtain
\begin{equation}\label{re-2.3}
\left\{
\aligned
\int_{\Delta_1} |\nabla u|^2 \, d\sigma 
& \le C \int_{\Delta_2 } \Big|\frac{\partial u}{\partial \nu} \Big|^2\, d\sigma
+ \int_{D_2} |\nabla u|^2\, dx,\\
\int_{\Delta_1} |\nabla u|^2 \, d\sigma
& \le C \int_{\Delta_2} |\nabla_{\tan}  u|^2\, d\sigma
+ C \int_{D_2} |\nabla u|^2\, dx,
\endaligned
\right.
\end{equation}
where $\Delta_r =\big\{ (x^\prime, \psi (x^\prime))\in \rd: |x^\prime|<r \text{ and } 
x_d =\psi(x^\prime) \big\}$.
We now take advantage of the fact that the dependence of $C$ on $\psi$
is only through $M$.
Since $\mathcal{L}_\varep (u_\varep)=0$
implies $\mathcal{L}_1 \big\{ u_\varep (\varep x) \big\}=0$,
 one may deduce from (\ref{re-2.3}) that
if $\mathcal{L}_\varep (u_\varep)=0$ in $D_{2\varep}$, then 
\begin{equation}\label{re-2.4}
\left\{
\aligned
\int_{\Delta_\varep} |\nabla u_\varep |^2 \, d\sigma 
& \le C \int_{\Delta_{2\varep} } \Big|\frac{\partial u_\varep }{\partial \nu_\varep} \Big|^2\, d\sigma
+ \frac{C}{\varep} \int_{D_{2\varep}} |\nabla u_\varep |^2\, dx,\\
\int_{\Delta_\varep} |\nabla u_\varep |^2 \, d\sigma
& \le C  \int_{\Delta_{2\varep}} |\nabla_{\tan}  u_\varep |^2\, d\sigma
+ \frac{C}{\varep} \int_{D_{2\varep}} |\nabla u_\varep |^2\, dx.
\endaligned
\right.
\end{equation}
Now, suppose that $u_\varep \in H^1(\Omega; \rd)$ and
$\mathcal{L}_\varep (u_\varep)=0$ in $\Omega$, where
$\Omega$ is a bounded Lipschitz domain in $\rd$.
By covering $\partial\Omega$ with a finite number of
suitable balls of size $c \varep$, it follows from (\ref{re-2.4}) that
\begin{equation}\label{re-2.5}
\left\{
\aligned
\int_{\partial\Omega} |\nabla u_\varep |^2 \, d\sigma 
& \le C \int_{\partial\Omega} \Big|\frac{\partial u_\varep }{\partial \nu_\varep} \Big|^2\, d\sigma
+ \frac{C}{\varep} \int_{\Omega_{c\varep}} |\nabla u_\varep |^2\, dx,\\
\int_{\partial\Omega} |\nabla u_\varep |^2 \, d\sigma
& \le C  \int_{\partial\Omega} |\nabla_{\tan}  u_\varep |^2\, d\sigma
+ \frac{C}{\varep} \int_{\Omega_{c\varep}} |\nabla u_\varep |^2\, dx.
\endaligned
\right.
\end{equation}
Notice that up to this point, we have only used the smoothness condition of $A$, not the periodicity of $A$.
With the additional periodicity condition we may invoke the estimates in Theorems 
\ref{main-theorem-1} and \ref{main-theorem-2} to bound the volume integrals of $|\nabla u_\varep|^2$
over the boundary layer $\Omega_{c\varep}$.
This yields the full Rellich estimates,
\begin{equation}\label{re-2.6}
\int_{\partial\Omega} |\nabla u_\varep |^2 \, d\sigma 
 \le C \int_{\partial\Omega} \Big|\frac{\partial u_\varep }{\partial \nu_\varep} \Big|^2\, d\sigma
\end{equation}
if $u_\varep \perp \mathcal{R}$, and 
\begin{equation}\label{re-2.7}
\int_{\partial\Omega} |\nabla u_\varep |^2 \, d\sigma
 \le C  \int_{\partial\Omega} |\nabla_{\tan}  u_\varep |^2\, d\sigma +C r_0^{-2} \int_{\partial\Omega} |u_\varep|^2\, d\sigma.
\end{equation}
It is well known that estimates (\ref{re-2.6})-(\ref{re-2.7}) may be used to solve the $L^2$ boundary value problems 
in Lipschitz domains by the method of layer potentials.
We refer the reader to \cite{Kenig-Shen-2} for the case where $A(y)$ satisfies (\ref{s-ellipticity}).
The details for the systems of elasticity will be carried out in a separate work \cite{Geng-Shen-Song-2}.
}
\end{remark}



\section{Convergence rates in $L^q$ for $q=\frac{2d}{d-1}$}
\setcounter{equation}{0}

In this section we establish sharp $O(\varep)$ estimates for $\| u_\varep -u_0\|_{L^q(\Omega)}$
with $q=\frac{2d}{d-1}$, using
Theorems \ref{main-theorem-1} and \ref{main-theorem-2} and a duality argument.
Throughout this section we will assume that $\Omega$ is a bounded Lipschitz domain and
$A=A(y)$ satisfies (\ref{ellipticity})-(\ref{periodicity}).

We start with the Dirichlet boundary condition.

\begin{lemma}\label{lemma-L-1}
Let $u_\varep$ ($\varep\ge 0$) be the solution of (\ref{DP-1}).
Suppose that $u_0\in H^2(\Omega; \mathbb{R}^d)$.
Then
\begin{equation}\label{L-1.0}
\| u_\varep -u_0 -\varep \chi_k(x/\varep) K_\varep \Big(\frac{\partial\widetilde{u}_0}{\partial x_k}\Big)
-v_\varep \|_{H^1_0(\Omega)}
\le C \, \varep \| \nabla^2 \widetilde{u}_0 \|_{L^2(\mathbb{R}^d)},
\end{equation}
where $\widetilde{u}_0 \in H^2(\mathbb{R}^d; \mathbb{R}^d)$ is an extension of $u_0$ and
$v_\varep \in H^1(\Omega; \mathbb{R}^d)$ is the weak solution to 
\begin{equation}\label{L-1.1}
\mathcal{L}_\varep (v_\varep)=0 \quad \text{ in } \Omega \quad \text{ and }\quad
v_\varep =-\varep \chi_k(x/\varep) K_\varep \Big(\frac{\partial\widetilde{u}_0}{\partial x_k}\Big)
\quad \text{ on }  \partial\Omega.
\end{equation}
\end{lemma}

\begin{proof}
Let
$$
w_\varep =u_\varep -u_0 -\varep \chi_k(x/\varep) K_\varep \Big(\frac{\partial\widetilde{u}_0}{\partial x_k}\Big)
-v_\varep.
$$
Using $\mathcal{L}_\varep (u_\varep) =\mathcal{L}_0 (u_0)$ and $\mathcal{L}_\varep (v_\varep)=0$
in $\Omega$,
a direct computation shows that
\begin{equation}\label{L-1.2}
\aligned
\mathcal{L}_\varep (w_\varep)
&=-\frac{\partial}{\partial x_i} 
\left\{ \Big[ \widehat{a}_{ij}^{\alpha\beta} -a_{ij}^{\alpha\beta} (x/\varep) \Big]
\frac{\partial u_0^\beta}{\partial x_j} \right\}
-\mathcal{L}_\varep \left\{ \varep \chi_k (x/\varep) K_\varep \Big(\frac{\partial \widetilde{u}_0}{\partial x_k} \Big) \right\}\\
&=-\frac{\partial }{\partial x_i}
\left\{ \Big[ \widehat{a}_{ij}^{\alpha\beta} -a_{ij}^{\alpha\beta} (x/\varep) \Big]
\Big[ \frac{\partial u_0^\beta}{\partial x_j}  -K_\varep \Big(\frac{\partial \widetilde{u}_0^\beta}{\partial x_j} \Big) \Big]\right\}\\
& \qquad +\frac{\partial}{\partial x_i} \left\{ b_{ij}^{\alpha\beta} (x/\varep)
K_\varep \Big( \frac{\partial \widetilde{u}_0^\beta}{\partial x_j}\Big) \right\}\\
&\qquad +\varep\, \frac{\partial}{\partial x_i}
\left\{ a_{ij}^{\alpha\beta} (x/\varep) \chi_k^{\beta\gamma} (x/\varep)
K_\varep \Big(\frac{\partial^2 \widetilde{u}_0^\gamma}{\partial x_j\partial x_k} \Big) \right\},
\endaligned
\end{equation}
where $b_{ij}^{\alpha\beta}$ is defined by (\ref{b}).
Using (\ref{phi}), we see that
\begin{equation}\label{L-1.3}
\frac{\partial}{\partial x_i}
\left\{ b_{ij}^{\alpha\beta} (x/\varep)
K_\varep \Big( \frac{\partial \widetilde{u}_0^\beta}{\partial x_j}\Big) \right\}
=-\varep\, \frac{\partial}{\partial x_i}
\left\{ \phi_{kij}^{\alpha\beta} (x/\varep) K_\varep \Big( \frac{\partial^2 \widetilde{u}_0^\beta}{\partial x_k
\partial x_j} \Big) \right\}.
\end{equation}
It follows from (\ref{L-1.2}) and (\ref{L-1.3}) by Lemmas \ref{lemma-2.00} and \ref{lemma-p} that
$$
\|\mathcal{L}_\varep (w_\varep) \|_{H^{-1}(\Omega)} 
\le C \, \| \nabla^2 \widetilde{u}_0 \|_{L^2(\mathbb{R}^d)},
$$
where $C$ depends only on $d$, $\kappa_1$, $\kappa_2$, and $\Omega$.
Since $w_\varep \in H_0^1(\Omega; \mathbb{R}^d)$,
this gives the estimate (\ref{L-1.0}) by the energy estimate.
\end{proof}

The following theorem establishes the sharp $O(\varep)$ estimate in $L^q$ with $q=\frac{2d}{d-1}$ for the
Dirichlet boundary condition.

\begin{theorem}\label{theorem-L-1}
Suppose that $A$ satisfies (\ref{ellipticity})-(\ref{periodicity}).
Let $\Omega$ be a bounded Lipschitz domain in $\mathbb{R}^d$.
Let $u_\varep$ ($\varep \ge 0$) be the weak solution to Dirichet problem (\ref{DP-1}).
Assume that $u_0 \in H^2(\Omega; \mathbb{R}^d)$. Then
\begin{equation}\label{L-2.0}
\| u_\varep -u_0 \|_{L^q(\Omega)} \le C\, \varep\,  \| u_0 \|_{H^2(\Omega)},
\end{equation}
where $q=\frac{2d}{d-1}$ and $C$ depends only on $d$, $\kappa_1$, $\kappa_2$, and $\Omega$.
\end{theorem}

\begin{proof}
We begin by choosing $\widetilde{u}_0\in H^2(\mathbb{R}^d; \mathbb{R}^d)$ such that
$\widetilde{u}_0 =u_0$ in $\Omega$ and $\| \widetilde{u}_0\|_{H^2(\mathbb{R}^d)}
\le C \| u_0\|_{H^2(\Omega)}$, where $C$ depends only on $\Omega$.
Since $\Omega$ is Lipschitz, this is possible by an extension theorem due to A. Calder\'on.
Next, since $H_0^1(\Omega) \subset L^q(\Omega)$ and
$$
\| \chi_k(x/\varep) K_\varep \Big(\frac{\partial \widetilde{u}_0 }{\partial x_k} \Big) \|_{L^q(\Omega)}
\le C\, \| \nabla \widetilde{u}_0 \|_{L^q(\mathbb{R}^d)}
\le C\, \| u_0 \|_{H^2(\Omega)},
$$
in view of Lemma \ref{lemma-L-1}, it suffices to show that
\begin{equation}\label{L-2.1}
\| v_\varep\|_{L^q(\Omega)} \le C\, \varep\, \| u_0\|_{H^2(\Omega)},
\end{equation}
where $v_\varep$ is given by (\ref{L-1.1}).

To this end we fix $G\in L^p(\Omega; \mathbb{R}^d)$, where $p=q^\prime=\frac{2d}{d+1}$, and let
$h_\varep \in H_0^1(\Omega; \mathbb{R}^d)$ be the weak solution to
\begin{equation}\label{DP-L}
\mathcal{L}_\varep (h_\varep) =G \quad \text{ in } \Omega \quad \text{ and } \quad 
h_\varep =0 \quad \text{ on } \partial\Omega.
\end{equation}
It follows from (\ref{L-1.1}), (\ref{DP-L}), and the divergence theorem that
$$
\aligned
\int_\Omega v_\varep \cdot G \, dx
&= -\int_{\partial\Omega} v_\varep \cdot \frac{\partial h_\varep}{\partial \nu_\varep}\, d\sigma \\
&=\varep \int_{\partial\Omega}  \chi_k (x/\varep) K_\varep 
\Big( \frac{\partial \widetilde{u}_0}{\partial x_k} \Big)
\cdot \frac{\partial h_\varep}{\partial \nu_\varep} (\eta_\varep-1) \, d\sigma \\
&=\int_\Omega \frac{\partial \chi_k^{\alpha\gamma}}{\partial x_i} (x/\varep)
K_\varep \Big( \frac{\partial \widetilde{u}_0^\gamma}{\partial x_k} \Big)a_{ij}^{\alpha\beta} (x/\varep) 
\frac{\partial h_\varep^\beta}{\partial x_j} (\eta_\varep-1)\, dx\\
&\qquad  +\varep \int_\Omega \chi_k^{\alpha\gamma} (x/\varep) K_\varep \Big( \frac{\partial^2 \widetilde{u}_0^\gamma}
{\partial x_i\partial x_k} \Big) a_{ij}^{\alpha\beta}(x/\varep)  
\frac{\partial h_\varep^\beta}{\partial x_j} (\eta_\varep-1)\, dx\\
&\qquad  -\varep \int_\Omega \chi_k^{\alpha\gamma} (x/\varep) K_\varep \Big( \frac{\partial \widetilde{u}_0^\gamma}
{\partial x_k}\Big) G^\alpha (\eta_\varep-1)\, dx\\
&\qquad +\varep \int_\Omega \chi_k^{\alpha\gamma} (x/\varep) K_\varep \Big( \frac{\partial \widetilde{u}_0^\gamma}
{\partial x_k} \Big) a_{ij}^{\alpha\beta}(x/\varep)  
\frac{\partial h_\varep^\beta}{\partial x_j} \frac{\partial \eta_\varep}{\partial x_i} \, dx,
\endaligned
$$
where $\eta_\varep \in C_0^\infty(\Omega)$ satisfies (\ref{eta}).
This implies that
\begin{equation}\label{L-2.2}
\aligned
\Big| \int_\Omega v_\varep \cdot G\, dx \Big|
&\le C \int_\Omega |\nabla \chi (x/\varep)|\,  |K_\varep (\nabla \widetilde{u}_0)|\,  |\nabla h_\varep|\,  |\eta_\varep-1|\, dx\\
&\qquad +C\varep  \int_\Omega | \chi (x/\varep)|\,  |K_\varep (\nabla^2 \widetilde{u}_0)| \, |\nabla h_\varep|\, |\eta_\varep-1|\, dx\\
& \qquad +C\varep  \int_\Omega | \chi (x/\varep)| \, |K_\varep (\nabla \widetilde{u}_0)| \, |G |\,  |\eta_\varep-1|\, dx\\
&\qquad+C\varep \int_\Omega | \chi (x/\varep)| \, |K_\varep (\nabla \widetilde{u}_0)|\,  |\nabla h_\varep|\,  |\nabla \eta_\varep|\, dx.
\endaligned
\end{equation}
Note that by Cauchy inequality and (\ref{eta}), the first and forth terms in the r.h.s. of (\ref{L-2.2}) are bounded by
$$
\aligned
& C \left( \int_{\Omega_{4\varep}} |(|\nabla \chi (x/\varep)| +|\chi (x/\varep)| ) K_\varep (\nabla \widetilde{u}_0)|^2\, dx\right)^{1/2}
\left( \int_{\Omega_{4\varep}} |\nabla h_\varep|^2\, dx \right)^{1/2}\\
&\qquad  \le C \left( \int_{\widetilde{\Omega}_{5\varep}}  |\nabla \widetilde{u}_0|^2\, dx\right)^{1/2}
\left( \int_{\Omega_{4\varep} }|\nabla h_\varep|^2 \, dx \right)^{1/2},
\endaligned
$$
where $\Omega_r =\big\{ x\in \Omega: \, \text{dist}(x, \partial\Omega)<r \big\}$,
$\widetilde{\Omega}_r =\big\{ x\in \mathbb{R}^d: \, \text{dist}(x, \partial\Omega)<r \big\}$, and
we have used Lemma \ref{lemma-2.00} for the last inequality.
Using the divergence theorem, as in (\ref{2.1-3}), one may prove that 
$$
\|\nabla \widetilde{u}_0 \|_{L^2(S_r)} \le C\,  \| \widetilde{u}_0 \|_{H^1(\mathbb{R}^d)}^{1/2}
\| \widetilde{u}_0 \|_{H^2(\mathbb{R}^d)}^{1/2},
$$
where $S_r=\big\{ x\in \mathbb{R}^d: \text{dist}(x, \partial\Omega)=r \big\}$.
It follows by the co-area formula that
\begin{equation}\label{Hardy}
\|\nabla \widetilde{u}_0 \|_{L^2(\widetilde{\Omega}_r)} \le C\,  r^{1/2}  
 \| \widetilde{u}_0 \|_{H^1(\mathbb{R}^d)}^{1/2}
\| \widetilde{u}_0 \|_{H^2(\mathbb{R}^d)}^{1/2}.
\end{equation}
This, together with the estimate in Theorem \ref{main-theorem-1} for $h_\varep$, shows that
the first and forth terms in the r.h.s. of (\ref{L-2.2}) are bounded by
$$
C\, \varep\, \| u_0\|_{H^2(\Omega)} \| G\|_{L^p(\Omega)},
$$
where $p=q^\prime =\frac{2d}{d+1}$.
Finally, we note that the second and third term in the r.h.s. of (\ref{L-2.2}) are 
bounded by
$$
\aligned
 &C\, \varep \, \| \nabla^2 \widetilde{u}_0 \|_{L^2(\mathbb{R}^d)} \|\nabla h_\varep\|_{L^2(\Omega)}
+C \, \varep\, \|\nabla \widetilde{u}_0 \|_{L^q(\mathbb{R}^d)} \| G\|_{L^p(\Omega)}\\
&\qquad \le C\, \varep\, \| u_0\|_{H^2(\Omega)} \| G\|_{L^p(\Omega)}.
\endaligned
$$
As a result, we have proved that
$$
\Big| \int_\Omega v_\varep \cdot G\, dx \Big|
 \le C\, \varep\, \| u_0\|_{H^2(\Omega)} \| G\|_{L^p(\Omega)},
$$
which, by duality, gives the estimate (\ref{L-2.1}) and completes the proof.
\end{proof}

Next we consider the solutions with the Neumann boundary conditions.

\begin{lemma}\label{lemma-L-2}
Let $u_\varep$ ($\varep\ge 0$) be the solutions of (\ref{NP-1}) such that 
$u_\varep \perp \mathcal{R}$.
Suppose that $u_0 \in  H^2(\Omega; \mathbb{R}^d)$.
Then
\begin{equation}\label{L-3.0}
\| u_\varep -u_0 -\varep \chi_k (x/\varep) K_\varep \Big( \frac{\partial \widetilde{u}_0}
{\partial x_k} \Big) -v_\varep \|_{H^1(\Omega)}
\le C\, \varep \Big\{ \| \nabla ^2 \widetilde{u}_0 \|_{L^2(\mathbb{R}^d)} +
\| \nabla  \widetilde{u}_0 \|_{L^2(\mathbb{R}^d)}\Big\},
\end{equation}
where
$\widetilde{u}_0$ is an extension of $u_0$ and $v_\varep\in H^1(\Omega; \mathbb{R}^d)$
is the weak solution to 
\begin{equation}\label{NP-L}
\left\{
\aligned
 &\mathcal{L}_\varep (v_\varep) =0 \quad \text{ in } \Omega,\\
& \frac{\partial v_\varep}{\partial \nu_\varep} 
= \frac{\varep}{2} \Big( n_k \frac{\partial}{\partial x_i}
-n_i \frac{\partial}{\partial x_k} \Big)
\left\{ \phi_{kij} (x/\varep) K_\varep \Big( \frac{\partial \widetilde{u}_0}{\partial x_j} \Big)\right\}
\quad \text{ on  } \partial\Omega,\\
&  v_\varep \perp \mathcal{R}.
\endaligned
\right.
\end{equation}
\end{lemma}

\begin{proof}
Let
$$
w_\varep =u_\varep -u_0 -\varep \chi_k(x/\varep) K_\varep \Big(\frac{\partial\widetilde{u}_0}{\partial x_k}\Big)
-v_\varep.
$$
Using $\frac{\partial u_\varep}{\partial\nu_\varep}
=\frac{\partial u_0}{\partial \nu_0}$ on $\partial\Omega$,
a direct computation shows that
\begin{equation}\label{L-3.1}
\aligned
\frac{\partial w_\varep}{\partial \nu_\varep}
&= \frac{\partial u_0}{\partial \nu_0} -\frac{\partial u_0}{\partial\nu_\varep}
-\frac{\partial}{\partial \nu_\varep} 
\left\{ \varep \chi_k (x/\varep) K_\varep \Big( \frac{\partial \widetilde{u}_0}{\partial x_k} \Big)\right\}
-\frac{\partial v_\varep}{\partial \nu_\varep}\\
&=n_i \Big[ \widehat{a}_{ij}^{\alpha\beta} -a_{ij}^{\alpha\beta} (x/\varep) \Big]
\Big[ \frac{\partial u_0^\beta}{\partial x_j} - K_\varep \Big( \frac{\partial u_0^\beta}{\partial x_j} \Big) \Big]\\
&\qquad
-n_i b_{ij}^{\alpha\beta}(x/\varep) K_\varep \Big( \frac{\partial u_0^\beta}{\partial x_j} \Big)\\
&\qquad  -n_i a_{ij}^{\alpha\beta} (x/\varep) \cdot \varep \chi_k^{\beta\gamma} (x/\varep)
K_\varep \Big( \frac{\partial^2 \widetilde{u}_0^\gamma}{\partial x_j \partial x_k} \Big) 
-\frac{\partial v_\varep}{\partial \nu_\varep}.
\endaligned
\end{equation}
Using (\ref{phi}), we also see that
\begin{equation}\label{L-3.2}
\aligned
 n_i b_{ij}^{\alpha\beta} & (x/\varep) K_\varep \Big( \frac{\partial \widetilde{u}_0^\beta}{\partial x_j} \Big)
+\frac{\partial v_\varep}{\partial \nu_\varep}\\
&=\varep n_i \frac{\partial}{\partial x_k} \Big[ \phi_{kij}^{\alpha\beta} (x/\varep) \Big]
K_\varep \Big( \frac{\partial \widetilde{u}_0^\beta}{\partial x_j} \Big)
+\frac{\partial v_\varep}{\partial \nu_\varep}\\
&=\frac{\varep}{2}\Big( n_i \frac{\partial}{\partial x_k} - n_k\frac{\partial}{\partial x_i} \Big)
\Big[ \phi_{kij}^{\alpha\beta} (x/\varep) \Big]
K_\varep \Big( \frac{\partial \widetilde{u}_0^\beta}{\partial x_j} \Big)
+\frac{\partial v_\varep}{\partial \nu_\varep}\\
&=-\varep n_i \phi_{kij}^{\alpha\beta} (x/\varep) K_\varep \Big( \frac{\partial^2 \widetilde{u}_0^\beta}{\partial x_k\partial x_j} \Big).
\endaligned
\end{equation}
As a result, we obtain 
\begin{equation}\label{L-3.3}
\aligned
\frac{\partial w_\varep}{\partial \nu_\varep}=
&n_i \Big[ \widehat{a}_{ij}^{\alpha\beta} -a_{ij}^{\alpha\beta} (x/\varep) \Big]
\Big[ \frac{\partial u_0^\beta}{\partial x_j} - K_\varep \Big( \frac{\partial u_0^\beta}{\partial x_j} \Big) \Big]\\
&\qquad
+\varep n_i \phi_{kij}^{\alpha\beta} (x/\varep) K_\varep \Big( \frac{\partial^2 \widetilde{u}_0^\beta}{\partial x_k\partial x_j} \Big)\\
&\qquad  -n_i a_{ij}^{\alpha\beta} (x/\varep) \cdot \varep \chi_k^{\beta\gamma} (x/\varep)
K_\varep \Big( \frac{\partial^2 \widetilde{u}_0^\gamma}{\partial x_j \partial x_k} \Big) .
\endaligned
\end{equation}

Next, we note that as in the proof of Lemma \ref{lemma-L-1},
\begin{equation}\label{L-3.4}
\aligned
\mathcal{L}_\varep (w_\varep)
&=-\frac{\partial }{\partial x_i}
\left\{ \Big[ \widehat{a}_{ij}^{\alpha\beta} -a_{ij}^{\alpha\beta} (x/\varep) \Big]
\Big[ \frac{\partial u_0^\beta}{\partial x_j}  -K_\varep \Big(\frac{\partial \widetilde{u}_0^\beta}{\partial x_j} \Big) \Big]\right\}\\
& \qquad -\varep\, \frac{\partial}{\partial x_i}
\left\{ \phi_{kij}^{\alpha\beta} (x/\varep) K_\varep \Big( \frac{\partial^2 \widetilde{u}_0^\beta}{\partial x_k
\partial x_j} \Big) \right\}
\\
&\qquad +\varep\, \frac{\partial}{\partial x_i}
\left\{ a_{ij}^{\alpha\beta} (x/\varep) \chi_k^{\beta\gamma} (x/\varep)
K_\varep \Big(\frac{\partial^2 \widetilde{u}_0^\gamma}{\partial x_j\partial x_k} \Big) \right\}.
\endaligned
\end{equation}
Thus, by (\ref{ellipticity}) and the energy estimate,
$$
\aligned
\|\nabla w_\varep  & +(\nabla w_\varep)^T \|_{L^2(\Omega)}\\
& \le C\, \|\nabla w_\varep\|_{L^2(\Omega)}
\Big\{ \|\nabla u_0 -K_\varep (\nabla \widetilde{u}_0)\|_{L^2(\Omega)}
+\varep \|\phi (x/\varep) K_\varep (\nabla^2 \widetilde{u}_0)\|_{L^2(\Omega)}\\
& \qquad \qquad \qquad  \qquad \qquad 
+\varep \| \chi (x/\varep) K_\varep (\nabla^2 u_0)\|_{L^2(\Omega)} \Big\}\\
&\le C \varep \|\nabla w_\varep\|_{L^2(\Omega)}
\| \nabla^2 \widetilde{u}_0\|_{L^2(\mathbb{R}^d)},
\endaligned
$$
where we have used Lemmas \ref{lemma-2.00} and \ref{lemma-p} for the last step.
By the second Korn inequality, this implies that
$$
\aligned
\| w_\varep\|_{H^1(\Omega)}
&\le C \, \varep \| \nabla^2 \widetilde{u}_0 \|_{L^2(\mathbb{R}^d)}
+ C \sum_{j=1}^J \Big| \int_\Omega w_\varep \cdot \phi_j \, dx \Big|\\
& \le C \, \varep \| \nabla^2 \widetilde{u}_0 \|_{L^2(\mathbb{R}^d)}
+ C \, \varep \| \chi(x/\varep) K_\varep (\nabla \widetilde{u}_0) \|_{L^2(\Omega)}\\
& \le C \, \varep \Big\{ \|\nabla^2 \widetilde{u}_0 \|_{L^2(\mathbb{R}^d)}
+\|\nabla \widetilde{u}_0 \|_{L^2(\mathbb{R}^d)}\Big\},
\endaligned
$$
where $\{ \phi_j: j=1, \dots, J \}$ forms an orthonormal basis of $\mathcal{R}$, as a subspace of $L^2(\Omega; \mathbb{R}^d)$.
The proof is complete.
\end{proof}

The next theorem is an analogue of Theorem \ref{theorem-L-1} for the Neumann boundary conditions.

\begin{theorem}\label{theorem-L-2}
Suppose that $A$ satisfies (\ref{ellipticity})-(\ref{periodicity}).
Let $\Omega$ be a bounded Lipschitz domain in $\mathbb{R}^d$.
Let $u_\varep$ ($\varep \ge 0$) be the weak solutions to the Neumann problem (\ref{NP-1})
with the property $u_\varep \perp \mathcal{R}$.
Assume that $u_0 \in H^2(\Omega; \mathbb{R}^d)$. Then
\begin{equation}\label{L-4.0}
\| u_\varep -u_0 \|_{L^q(\Omega)} \le C\, \varep\,  \| u_0 \|_{H^2(\Omega)},
\end{equation}
where $q=\frac{2d}{d-1}$ and $C$ depends only on $d$, $\kappa_1$, $\kappa_2$, and $\Omega$.
\end{theorem}

\begin{proof}
As in the proof of Theorem \ref{theorem-L-1}, it suffices to show that
\begin{equation}\label{L-4.1}
\| v_\varep\|_{L^q(\Omega)} 
\le C\, \varep\,  \| u_0 \|_{H^2(\Omega)},
\end{equation}
where $v_\varep$ is given by (\ref{NP-L}).
To this end we fix $G\in L^p(\Omega; \mathbb{R}^d)$ with $G\perp \mathcal{R}$ and
let $h_\varep\in H^1(\Omega; \mathbb{R}^d)$ be the weak solution to
\begin{equation}\label{L-4.2}
\mathcal{L}_\varep (h_\varep)=G \quad \text{ in } \Omega \quad \text{ and } \quad 
\frac{\partial h_\varep}{\partial \nu_\varep} =0 \quad \text{ on } \partial\Omega,
\end{equation}
with the property $h_\varep \perp \mathcal{R}$.
It follows from (\ref{L-4.2}), (\ref{NP-L}), and the Green's formula that
$$
\aligned
\int_\Omega v_\varep \cdot G\, dx
&=\int_\Omega A(x/\varep) \nabla v_\varep \cdot \nabla h_\varep\, dx
=\int_{\partial \Omega} \frac{\partial v_\varep}{\partial \nu_\varep} \cdot h_\varep\, d\sigma\\
&=\frac{\varep}{2}
\int_{\partial \Omega} \Big( n_k \frac{\partial}{\partial x_i}
-n_i \frac{\partial}{\partial x_k} \Big)
\left\{ \phi_{kij}^{\alpha\beta} (x/\varep) K_\varep \Big( \frac{\partial \widetilde{u}_0^\beta}{\partial x_j}\Big) \right\}
\cdot h_\varep^\alpha \, d\sigma\\
&=-\frac{\varep}{2}
\int_{\partial \Omega}\phi_{kij}^{\alpha\beta} (x/\varep) K_\varep \Big( \frac{\partial \widetilde{u}_0^\beta}{\partial x_j}\Big)
\cdot \Big( n_k \frac{\partial}{\partial x_i}
-n_i \frac{\partial}{\partial x_k} \Big) h_\varep^\alpha \cdot (1-\eta_\varep)\, d\sigma\\
&=-{\varep}\int_\Omega
\frac{\partial}{\partial x_k}
\left\{ \phi_{kij}^{\alpha\beta} (x/\varep) K_\varep \Big( \frac{\partial \widetilde{u}_0^\beta}{\partial x_j}\Big) (1-\eta_\varep) \right\}
\cdot \frac{\partial h_\varep^\alpha}{\partial x_i} \, dx,
\endaligned
$$
where $\eta_\varep \in C_0^\infty(\Omega)$ satisfies (\ref{eta}) and we have used the divergence theorem
as well as (\ref{phi}) for the last inequality.
This leads to
\begin{equation}\label{L-4.3}
\aligned
\Big| \int_\Omega v_\varep \cdot G\, dx \Big|
&\le C \int_{\Omega_{4\varep}} |\nabla \phi (x/\varep)|\, | K_\varep (\nabla \widetilde{u}_0)|\, |\nabla h_\varep|\, dx\\
& \qquad +C\, \varep \int_{\Omega_{4\varep}} |\phi(x/\varep)|\, |K_\varep (\nabla^2 \widetilde{u}_0)|\, |\nabla h_\varep|\, dx \\
& \qquad + C \,\varep \int_{\Omega_{4\varep}} |\phi(x/\varep)|\, |K_\varep (\nabla \widetilde{u}_0)|
\, |\nabla \eta_\varep|\, |\nabla h_\varep|\, dx.
\endaligned
\end{equation}
Note that by the Cauchy inequality, the first and third term in the r.h.s. of (\ref{L-4.3}) are bounded by
$$
\aligned
& C \| \Big( |\nabla \phi(x/\varep)| + |\phi(x/\varep)| \Big) K_\varep (\nabla \widetilde{u}_0) \|_{L^2(\Omega_{4\varep})}
\| \nabla h_\varep \|_{L^2(\Omega_{4\varep})} \\
&\qquad \le C \, \|\nabla \widetilde{u}_0 \|_{L^2(\widetilde{\Omega}_{5\varep})} \|\nabla h_\varep \|_{L^2(\Omega_{4\varep})}\\
&\qquad \le C\, \varep\,  \| u_0\|_{H^2(\Omega)} \| G\|_{L^p(\Omega)},
\endaligned
$$
where we have used Lemma \ref{lemma-p} for the first inequality and Theorem \ref{main-theorem-2} as well as
estimate (\ref{Hardy}) for the second.
Also, the second term in the r.h.s. of (\ref{L-4.3}) is bounded by
$$
\aligned
& C\, \varep\, \|\phi(x/\varep) K_\varep (\nabla^2 \widetilde{u}_0)\|_{L^2(\Omega)} \|\nabla h_\varep\|_{L^2(\Omega)}\\
&\qquad \le C\, \varep\, \| u_0\|_{H^2(\Omega)} \| G\|_{L^p(\Omega)}.
\endaligned
$$
Hence we have proved that for any $G\in L^p(\Omega; \mathbb{R}^d)$ with the property $G\perp \mathcal{A}$,
$$
\Big| \int_\Omega v_\varep \cdot G\, dx \Big|
\le C\, \varep \, \| u_0\|_{H^2(\Omega)} \| G\|_{L^p(\Omega)}.
$$
Since $v_\varep \perp \mathcal{A}$,
this gives the estimate (\ref{L-4.1}) by duality and completes the proof.
\end{proof}

Note that by combining Theorems \ref{theorem-L-1} and \ref{theorem-L-2},
one obtains Theorem \ref{main-theorem-L}.



\section{$C^\alpha$ estimates in $C^1$ domains}
\setcounter{equation}{0}

In this section we investigate uniform boundary $C^\alpha$ estimates in
$C^1$ domains. The results will be used in the next section to establish uniform boundary
$W^{1,p}$ estimates in $C^1$ domains. 
Throughout the section we will assume that the defining function $\psi$ in $D_r$ and $\Delta_r$
is $C^1$ and $\psi(0)=0$. To quantify the $C^1$ condition we further assume that
\begin{equation}\label{tau}
\sup \Big\{ |\nabla \psi (x^\prime) -\nabla \psi (y^\prime)|: x^\prime, y^\prime \in \mathbb{R}^{d-1} 
\text{ and } |x^\prime -y^\prime|\le t \Big\}
\le \tau (t),
\end{equation} 
where $\tau (t)\to 0$ as $t\to 0^+$.

The rescaling argument is used frequently in this paper. 
Suppose that $\mathcal{L}_\varep (u_\varep)=F$ in $D_{2r}$ and $u_\varep=f$ on $\Delta_{2r}$.
Let  $w(x)= u_\varep (rx)$. Then 
$$
\mathcal{L}_{\frac{\varep}{r}} (w)= G \quad \text{ in } \widetilde{D}_2 \quad 
\text{ and } \quad w=g \quad \text{ on } \widetilde{\Delta}_2,
$$
where $G(x)=r^2 F(rx)$, $g(x)= f(rx)$, and
$$
\aligned
\widetilde{D}_2  &=\big\{ (x^\prime, x_d) \in \rd: |x^\prime|<2 \text{ and } \psi_r (x^\prime)<x_d <\psi_r (x^\prime) +2 \big\},\\
\widetilde{\Delta}_2 &=\big\{ (x^\prime, x_d) \in \rd: |x^\prime|<2 \text{ and } x_d =\psi_r (x^\prime) \big\},
\endaligned
$$
with $\psi_r (x^\prime)=r^{-1} \psi(rx^\prime)$.
Note that $\psi_r (0)=0$ and $\|\nabla \psi_r\|_\infty=\|\nabla \psi\|_\infty$.
Moreover, if $\psi$ is $C^1$ and satisfies (\ref{tau}), then
$\psi_r$ satisfies (\ref{tau}) uniformly in $r$ for $0<r\le 1$.

\begin{lemma}\label{lemma-6.1}
Let $0<\varep\le r\le 1$.
Let $u_\varep \in H^1(D_{2r};\rd)$ be a weak solution of $\mathcal{L}_\varep (u_\varep) =0$
in $D_{2r}$ with $u_\varep =0$ on $\Delta_{2r}$.
Then there exists $v\in H^1(D_r; \rd)$ such that $\mathcal{L}_0 (v)=0$ in $D_r$,
$v=0$ on $\Delta_r$, and
\begin{equation}\label{6.1-0}
\left(\average_{D_r} |u_\varep -v|^2\right)^{1/2}
\le C\, \left(\frac{\varep}{r} \right)^{1/2}
\left(\average_{D_{2r}} |u_\varep|^2 \right)^{1/2},
\end{equation}
where $\|\nabla \psi\|_\infty\le M$, and
$C$ depends only on $d$, $\kappa_1$, $\kappa_2$, and $M$.
\end{lemma}

\begin{proof}
By rescaling we may assume $r=1$.
By Cacciopoli's inequality, 
\begin{equation}\label{C-4-D}
\left(\average_{D_{3/2} } |\nabla u_\varep|^2\right)^{1/2}
\le C \left(\average_{D_2} |u_\varep|^2 \right)^{1/2}.
\end{equation}
It follows from (\ref{C-4-D}) and the co-area formula that there exists $t\in [4/5, 3/2]$ such that
\begin{equation}\label{6.1-1}
\| \nabla u_\varep \|_{L^2(\partial D_t\setminus \Delta_2)}  +\| u_\varep \|_{L^2(\partial D_t\setminus D_2)}
\le C \, \| u_\varep\|_{L^2(D_2)}.
\end{equation}
Let $v$ be the solution to the Dirichlet problem:
$\mathcal{L}_0 (v)=0$ in $D_t$ and $v=u_\varep$ on $\partial D_t$.
Note that $v=0$ on $\Delta_1$, and by Remark \ref{remark-2.1},
\begin{equation}\label{6.1-2}
\| u_\varep -v\|_{L^2(D_t)}
\le C\varep^{1/2} \, \| u_\varep \|_{H^1(\partial D_t)}.
\end{equation}
This, together with (\ref{6.1-1}), gives
$$
\| u_\varep - v\|_{L^2(D_1)}
\le \| u_\varep -v\|_{L^2(D_t)}
\le C \varep^{1/2}\,  \| u_\varep\|_{L^2(D_2)},
$$
and completes the proof.
\end{proof}

\begin{theorem}\label{theorem-6.1}
Suppose that $A=A(y)$ satisfies (\ref{ellipticity})-(\ref{periodicity}).
Let $u_\varep$ be a weak solution of $\mathcal{L}_\varep (u_\varep) =0$ in $D_1$
with $u_\varep =0$ on $\Delta_1$, where the defining function $\psi$ in $D_1$ and $\Delta_1$
is $C^1$. Then, for any $\alpha \in (0,1)$ and $\varep\le r\le (1/2)$,
\begin{equation}\label{6.2-0}
\left(\average_{D_r} |\nabla u_\varep|^2 \right)^{1/2}
\le C_\alpha\, r^{\alpha -1} \left(\average_{D_1} |u_\varep|^2\right)^{1/2},
\end{equation}
where $C_\alpha$ depends only on $d$, $\alpha$, $\kappa_1$, $\kappa_2$, and 
the function $\tau (t)$ in (\ref{tau}).
\end{theorem}

\begin{proof} Fix $\beta \in (\alpha,1)$.
For each $r\in [\varep, 1/2]$, let $v=v_r$ be the function given by Lemma \ref{lemma-6.1}.
By the boundary $C^\beta$ estimates in $C^1$ domains for the operator $\mathcal{L}_0$,
$$
\left(\average_{D_{\theta r} } |v|^2 \right)^{1/2}
\le C_0 \theta^\beta \left(\average_{D_r} |v|^2\right)^{1/2},
$$
for any $\theta \in (0,1)$, where $C_0$ depends only on $d$, $\kappa_1$, $\kappa_2$, $\beta$ and $\tau (t)$.
It follows that
$$
\aligned
\left(\average_{D_{\theta r}} |u_\varep|^2 \right)^{1/2}
&\le \left(\average_{D_{\theta r}} | v|^2 \right)^{1/2}
+ C \left(\average_{D_{\theta r}} |u_\varep -v |^2 \right)^{1/2}\\
&\le C\, \theta^\beta  \left(\average_{D_{ r}} | v|^2 \right)^{1/2}
+ C\theta^{-\frac{d}{2}} \left(\average_{D_{r}} |u_\varep -v |^2 \right)^{1/2}\\
 & \le  C_1\, \theta^\beta \left(\average_{D_{ r}} | u_\varep |^2 \right)^{1/2}
+ C_1\, \theta^{-\frac{d}{2}} \left(\frac{\varep}{r}\right)^{1/2}\left(\average_{D_{2r}} |u_\varep  |^2 \right)^{1/2},
\endaligned
$$
for any $\varep\le r\le 1/2$.
We now choose $\theta\in (0,1/4)$ so  small that $C_1\theta^{\beta-\alpha}<(1/4)$.
With $\theta$ fixed, choose $N>1 $ large so that 
$$
C_1 2^\alpha \theta^{-\frac{d}{2}-\alpha } N^{-1/2} \le (1/4).
$$
It follows that if $r\ge N\varep$,
\begin{equation}\label{6.2-10}
\phi (\theta r) \le \frac14 \Big\{ \phi (r) +\phi (2r) \Big\},
\end{equation}
where
$$
\phi (r) =r^{-\alpha} \left(\average_{D_r} |u_\varep|^2 \right)^{1/2}.
$$
By integration we may deduce from (\ref{6.2-10}) that
$$
\int_{\theta a}^{\theta/2} \phi (r) \frac{dr}{r}
\le \frac14 \int_{a}^{1/2}  \phi (r) \frac{dr}{r}
+\frac14  \int_{2a}^{1}  \phi (r) \frac{dr}{r},
$$
where $N\varep \le a<  (1/2)$. This implies that
$$
\int_{\theta a}^{1} \phi (r) \, \frac{dr}{r}
\le C \int_{ \theta/2}^{1} \phi( r)\, \frac{dr}{r}\le C \, \phi (1).
$$
Hence, $\phi (r)\le C\, \phi (1)$ for any $r\in [\varep, 1]$, and 
the estimate (\ref{6.2-0}) now follows by Cacciopoli's inequality.
\end{proof}

\begin{remark} 
{\rm
Under the stronger assumption that the defining function $\phi$ for $D_1$ is $C^{1, \sigma}$
for some $\sigma>0$, we will show in Section 8 that the estimate (\ref{6.2-0}) holds for $\alpha=1$.
In particular, it follows from the argument in Section 7 that if $\mathcal{L}_\varep (u_\varep)=0$
in $B(0,1)$, then
\begin{equation}\label{I-6}
\left(\average_{B(0,r)} |\nabla u_\varep|^2\right)^{1/2}
\le C \, \left(\average_{B(0,1)} |\nabla u_\varep|^2\right)^{1/2}
\end{equation}
for any $\varep\le r< 1$.
This is the interior Lipschitz estimate down the scale $\varep$.
}
\end{remark}

A function $A$ is said to belong to $V\!M\!O(\rd)$ if the l.h.s. of
(\ref{VMO})  goes to zero as $t\to 0^+$. To quantify this assumption we assume that
\begin{equation}\label{VMO}
\sup_{\substack{x\in \rd\\ 0<r<t}}
\average_{B(x,r)} \Big| A (y) -\average_{B(x,r)} A \Big|\, dy \le \rho (t),
\end{equation}
where $\rho (t) \to 0$ as $t \to 0^+$.

The following corollary was essentially proved in \cite{AL-1987} by a compactness method.

\begin{corollary}\label{cor-6-1}
Suppose that $A$ satisfies (\ref{ellipticity})-(\ref{periodicity}).
Also assume that $A\in V\!M\!O(\rd)$. 
Let $u_\varep\in H^1(D_1;\rd)$ be a weak solution of $\mathcal{L}_\varep (u_\varep)=0$
in $D_1$ with $u_\varep =0$ on $\Delta_1$.
Then, for any $\alpha\in (0,1)$,
\begin{equation}\label{6.3-0}
\| u_\varep\|_{C^\alpha (D_{1/2})}
\le C_\alpha  \left(\average_{D_1} |u_\varep|^2 \right)^{1/2},
\end{equation}
where $C_\alpha$ depends only on $d$, $\kappa_1$, $\kappa_2$, $\alpha$, and 
the functions $\tau (t)$, $\rho (t)$.
\end{corollary}

\begin{proof}
We may assume that $0<\varep<(1/2)$, as the case of $\varep\ge (1/2)$ is local.
Since $\mathcal{L}_1\big( u_\varep (\varep x) \big)=0$,
it follows from the boundary $C^\alpha$ estimates in $C^1$ domains
for the operator $\mathcal{L}_1$ by rescaling that if  $\alpha  \in (0,1)$ and $0<r<\varep$,
$$
\left(\average_{D_r} |\nabla u_\varep|^2\right)^{1/2}
\le C \left(\frac{r}{\varep}\right)^{\alpha -1} \left(\average_{D_\varep} |\nabla u_\varep|^2\right)^{1/2},
$$
where $C$ depends only on $d$, $\kappa_1$, $\kappa_2$, $\alpha$,
$\tau (t) $ and $\rho (t)$.
This, together with Theorem \ref{theorem-6.1}, shows that the estimate
(\ref{6.2-0}) holds for any $0<r<(1/2)$.
By combining (\ref{6.2-0}) with a similar interior estimate, we obtain 
\begin{equation}\label{6.3-1}
r^{\alpha-1} \left(\average_{B(x, r)\cap D_{1/2}} |\nabla u_\varep|^2\right)^{1/2}
\le C\,  \| u_\varep\|_{L^2(D_1)},
\end{equation}
for any $0<r<c$ and $x\in D_{1/2}$.
The estimate (\ref{6.3-0}) follows from (\ref{6.3-1}) by  Campanato's characterization of H\"older spaces.
\end{proof}

The rest of this section is devoted to the boundary $C^\alpha$
estimates for solutions with the Neumann boundary conditions.

\begin{lemma}\label{lemma-6.4}
Let $0<\varep\le r\le 1$.
Let $u_\varep\in H^1(D_{2r};\rd)$ be a weak solution of $\mathcal{L}_\varep (u_\varep)=0$
in $D_{2r}$ with $\frac{\partial u_\varep}{\partial \nu_\varep} =0$ on $\Delta_{2r}$.
Then there exists a function $w\in H^1(D_r; \rd)$ such that
$\mathcal{L}_0 (w)=0$, $\frac{\partial w}{\partial \nu_0}=0$ in $\Delta_r$,
and
\begin{equation}\label{6.4}
\left(\average_{D_r} |u_\varep -w|^2\right)^{1/2}
\le C \left(\frac{\varep}{r} \right)^{1/2}
\left(\average_{D_{2r}} |u_\varep|^2 \right)^{1/2},
\end{equation}
where $\|\psi\|_\infty\le M$, and
$C$ depends only on $d$, $\kappa_1$, $\kappa_2$, and $M$.
\end{lemma}

\begin{proof}
By rescaling we may assume $r=1$.
As in the proof of Lemma \ref{lemma-6.1}, there exists $t\in [4/5, 3/2]$ such that
$$
\| u_\varep\|_{L^2(\partial D_t\setminus \Delta_2 )} +\|\nabla u_\varep\|_{L^2(\partial D_t\setminus \Delta_2)}
\le C\, \| u_\varep\|_{L^2(D_2)}.
$$
Let $\phi_\varep$ be a function in $\mathcal{R}$ such that
$u_\varep -\phi_\varep \perp \mathcal{R}$ in $L^2(D_t; \rd)$.
Let $v$ be the solution to the Neumann problem:
$\mathcal{L}_0 (v)=0$ in $D_t$ and $\frac{\partial v}{\partial \nu_0}=
\frac{\partial u_\varep}{\partial\nu_\varep}$ on $\partial D_t$, with
$v \perp \mathcal{R}$.
It follows from Remark \ref{remark-2.1} that
$$
\aligned
\|u_\varep -\phi_\varep -v\|_{L^2(D_1)}
&\le \|u_\varep -\phi_\varep -v\|_{L^2(D_t)}\\
&\le C \varep^{1/2} \|\frac{\partial u_\varep}{\partial\nu_\varep} \|_{L^2(\partial D_t)}\\
&\le C\varep^{1/2} \| u_\varep\|_{L^2(D_2)}.
\endaligned
$$
It is easy to see that the function $w=v+\phi_\varep$ satisfies the desired conditions.
\end{proof}

\begin{theorem}\label{theorem-6.2}
Suppose that $A=A(y)$ satisfies (\ref{ellipticity})-(\ref{periodicity}).
Let $u_\varep$ be a weak solution of $\mathcal{L}_\varep (u_\varep) =0$ in $D_1$
with $\frac{\partial u_\varep}{\partial\nu_\varep} =0$ on $\Delta_1$, where the defining function $\psi$ in $D_1$ and $\Delta_1$
is $C^1$. Then, for any $\alpha \in (0,1)$ and $\varep\le r\le 1$,
\begin{equation}\label{6.5-0}
\left(\average_{D_r} |\nabla u_\varep|^2 \right)^{1/2}
\le C_\alpha \, r^{\alpha -1} \left(\average_{D_1} |\nabla u_\varep|^2\right)^{1/2},
\end{equation}
where $C$ depends only on $d$, $\alpha$, $\kappa_1$, $\kappa_2$, and the function $\tau (t)$.
\end{theorem}

\begin{proof}
Fix $\beta\in (\alpha,1)$.
For each $r\in [\varep, 1/2]$, let $w=w_r$ be the function given by Lemma \ref{lemma-6.4}.
By the boundary $C^\beta$ estimates in $C^1$ domains for the operator $\mathcal{L}_0$,
$$
\inf_{q\in \rd} \left(\average_{D_{\theta r}} | w-q|^2 \right)^{1/2}
\le C_0 \theta^\beta  \inf_{q\in \rd} 
\left(\average_{D_r} |w -q|^2 \right)^{1/2},
$$
where $C_0$ depends only on $d$, $\beta$, $\kappa_1$, $\kappa_2$, and $\tau (t)$.
This, together with Lemma \ref{lemma-6.4}, gives
$$
\aligned
\inf_{q\in \rd}  & \left(\average_{D_{\theta r}} | u_\varep-q|^2 \right)^{1/2}\\
&\le C \inf_{q\in \rd} \left(\average_{D_{\theta r}} | w -q |^2\right)^{1/2}
+\left(\average_{D_{\theta r}} | u_\varep -w |^2 \right)^{1/2}\\
&\le C\theta^\beta  \inf_{q\in \rd} \left(\average_{D_{r}} | w -q |^2\right)^{1/2}
+C_0 \theta^{-\frac{d}{2}} 
\left(\average_{D_{r}} |u_\varep -w|^2\right)^{1/2}\\
& \le C\theta^\beta  \inf_{q\in \rd} \left(\average_{D_{r}} | u_\varep  -q |^2\right)^{1/2}
+C \theta^{-\frac{d}{2}} \left(\frac{\varep}{r}\right)^{1/2}
\left(\average_{D_{2r}} |u_\varep |^2\right)^{1/2}
\endaligned
$$
By replacing $u_\varep$ with $u_\varep-q$, we obtain 
$$
\phi (\theta r) \le C \theta^{\beta -\alpha} \phi (r) + C \theta^{-\alpha-\frac{d}{2} } (\varep/r)^{1/2} \phi (2r)
$$
for any $r\in [\varep, 1/2]$,
where
$$
\phi(r)= r^{-\alpha} \inf_{q\in \rd}   \left(\average_{D_{ r}} | u_\varep-q|^2 \right)^{1/2}.
$$
By the integration argument used in the proof of Theorem \ref{theorem-6.1},
we may conclude that $\phi (r)\le C\phi (1) $ for $r\in[\varep, 1/2]$, which yields (\ref{6.5-0})
by Cacciopoli's inequality.
\end{proof}

\begin{remark}
{\rm 
Under the stronger condition that the defining function for $D_1$ and $\Delta_1$ is $C^{1, \sigma}$
for some $\sigma>0$, we will show in Section 9 that the estimate (\ref{6.5-0}) holds for $\alpha=1$.
}
\end{remark}

The following corollary was essentially proved in \cite{KLS1} by a compactness method.

\begin{corollary}\label{cor-6-2}
Suppose that $A$ satisfies (\ref{ellipticity})-(\ref{periodicity}).
Also assume that $A\in V\!M\!O(\rd)$. 
Let $u_\varep\in H^1(D_1; \rd)$ be a weak solution of $\mathcal{L}_\varep (u_\varep)=0$
in $D_1$ with $\frac{\partial u_\varep}{\partial \nu_\varep} =0$ on $\Delta_1$.
Then, for any $\alpha\in (0,1)$,
\begin{equation}\label{6.6-0}
\| u_\varep\|_{C^\alpha (D_{1/2})}
\le C_\alpha  \left(\average_{D_1} |u_\varep|^2 \right)^{1/2},
\end{equation}
where $C_\alpha$ depends only on $d$, $\kappa_1$, $\kappa_2$, $\alpha$,
and the functions $\tau (t)$, $\rho (t)$.
\end{corollary}

\begin{proof}
As in the case of the Dirichlet boundary condition, the additional smoothness assumption $A\in V\!M\!O(\rd)$
ensures that the estimates (\ref{6.5-0}) holds for any $r\in (0,1/2)$.
This, together with  the interior estimates,
gives the estimate (\ref{6.6-0}) by the use of Campanato's characterization of H\"older spaces.
\end{proof}



\section{$W^{1,p}$ estimates in $C^1$ domains}
\setcounter{equation}{0}

In this section we study the uniform $W^{1,p}$ estimates in $C^1$ domains.
Throughout the section we will assume that
$A=A(y)$ satisfies (\ref{ellipticity})-(\ref{periodicity}), $A\in V\!M\!O(\rd)$, and
$\Omega$ is $C^1$.
Our goal is to prove the following two theorems.

\begin{theorem}\label{theorem-7.1}
Suppose that $A$ satisfies (\ref{ellipticity})-(\ref{periodicity}).
Also assume that $A\in V\!M\!O(\rd)$.
Let $1<p<\infty$ and $\Omega$ be a bounded $C^1$ domain in $\rd$.
Let $u_\varep\in W^{1, p}(\Omega; \mathbb{R}^d)$ be a weak solution to the Dirichlet problem
\begin{equation}\label{DP-7}
\mathcal{L}_\varep (u_\varep) =\text{div} (f) \quad \text{ in } \Omega
\quad \text{ and } \quad u_\varep =0 \quad  \text{ on } \partial\Omega,
\end{equation}
where $f=(f_i^\alpha)\in L^p(\Omega; \mathbb{R}^{d\times d})$.
Then
\begin{equation}\label{7.1-0}
\| u_\varep\|_{W^{1,p}(\Omega)} \le C_p\, \| f\|_{L^p(\Omega)}
\end{equation}
where $C_p$ depends only on $d$, $p$, $A$, and $\Omega$.
\end{theorem}

\begin{theorem}\label{theorem-7.2}
Suppose that $A$ satisfies the same conditions as in Theorem \ref{theorem-7.1}.
Let $1<p<\infty$ and $\Omega$ be a bounded $C^1$ domain in $\rd$.
Let $u_\varep\in W^{1,p}(\Omega; \rd)$ be a weak solution to the Neumann problem
\begin{equation}\label{NP-7}
\mathcal{L}_\varep (u_\varep) =\text{div} (f) \quad \text{ in } \Omega
\quad \text{ and } \quad  \frac{\partial u_\varep}{\partial \nu_\varep}
=-n\cdot f \quad \text{ on } \partial\Omega,
\end{equation}
where $f=(f_i^\alpha)\in L^p(\Omega; \mathbb{R}^{d\times d})$.
Assume that $u_\varep \perp \mathcal{R}$.
Then
\begin{equation}\label{7.2-0}
\|u_\varep\|_{W^{1,p}(\Omega)} \le C_p\,  \| f\|_{L^p(\Omega)},
\end{equation}
where $C_p$ depends only on $d$, $p$, $A$, and $\Omega$.
\end{theorem}

Recall that a function $u_\varep$ is called a weak solution of (\ref{DP-7}) if 
$u_\varep\in W_0^{1,p}(\Omega; \rd)$ and
\begin{equation}\label{weak-solution}
\int_\Omega a_{ij}^{\alpha\beta} (x/\varep) \frac{\partial u_\varep^\beta}
{\partial x_j}\cdot \frac{\partial \varphi^\alpha}{\partial x_i}\, dx
=-\int_\Omega f_i^\alpha \cdot \frac{\partial \varphi^\alpha}{\partial x_i} \, dx
\end{equation}
for any $\varphi =(\varphi^\alpha) \in C_0^\infty(\Omega; \rd)$.
Similarly, $u_\varep$ is called a weak solution of (\ref{NP-7})
if $u_\varep \in W^{1,p}(\Omega; \rd)$
and  (\ref{weak-solution}) holds for any $\varphi =(\varphi^\alpha) \in C_0^\infty(\rd; \rd)$.
Under the assumptions that $A\in V\!M\!O(\rd)$ and $\Omega$ is $C^1$,
the existence and uniqueness of solutions of (\ref{DP-7}) and (\ref{NP-7})
are more or less well known (see e.g. \cite{Byun-Wang-2004,Byun-Wang-2005} for references).
The main interest here is that the constants $C$ in the $W^{1, p}$ estimates
(\ref{7.1-0}) and (\ref{7.2-0}) are independent of  $\varep$.
We mention that for $\mathcal{L}_\varep$ with coefficients satisfying (\ref{periodicity}), (\ref{s-ellipticity})
and H\"older continuity condition, estimates (\ref{7.1-0}) and
(\ref{7.2-0}) were established in \cite{AL-1987, AL-1991, Shen-2008, KLS1}.
The results were extended to the case of almost-periodic coefficients in \cite{Armstrong-Shen-2014}.
Also,
for $\mathcal{L}_\varep$ with coefficients satisfying (\ref{ellipticity})-(\ref{periodicity})
in Lipschitz domains, some partial results may be found in \cite{GSS}.

Theorems \ref{theorem-7.1} and \ref{theorem-7.2} are proved by a real-variable argument.
The required weak reverse H\"older inequalities  (\ref{7.0-0})
and (\ref{7.1-0}) for $p>2$
are established by combining local estimates for $\mathcal{L}_1$ and
boundary H\"older estimates in Section 4 with the interior Lipschitz estimates,
up to the scale $\varep$.

\begin{lemma}\label{lemma-7.0}
Let $u_\varep\in H^1(B(x_0,2r); \rd)$ be a weak solution to $\mathcal{L}_\varep (u_\varep)=0$
in $B(x_0, 2r)$ for some $x_0\in \rd$ and $r>0$.
Then, for any $2<p<\infty$,
\begin{equation}\label{7.0-0}
\left(\average_{B(x_0,r)} |\nabla u_\varep|^p\right)^{1/p}
\le C_p \left(\average_{B(x_0,2r)} |\nabla u_\varep|^2\right)^{1/2},
\end{equation}
where $C_p$ depends only on $d$,  $p$, $\kappa_1$, $\kappa_2$, and the function $\rho(t)$ in (\ref{VMO}).
\end{lemma}

\begin{proof}
By translation and dilation we may assume that $x_0=0$ and $r=1$.
We may also assume that $0<\varep<(1/4)$.
The case $\varep\ge (1/4)$ for $B(0,1)$ is local, since $A(x/\varep)$ satisfies the smoothness 
condition (\ref{VMO}) uniformly in $\varep$.
For each $y\in B(0,1)$, we use the local $W^{1,p}$ estimates for
the operator $\mathcal{L}_1$ and a blow-up argument to show that
\begin{equation}\label{7.0-1}
\left(\average_{B(y, \varep/2)} |\nabla u_\varep|^p\right)^{1/p}
\le C \left(\average_{B(y, \varep)} |\nabla u_\varep|^2 \right)^{1/2}.
\end{equation}
By the interior Lipschitz estimate, up to the scale $\varep$, we have
\begin{equation}\label{7.0-2}
\left(\average_{B(y, \varep)} |\nabla u_\varep|^2 \right)^{1/2}
\le C \left(\average_{B(y, 1)} |\nabla u_\varep|^2 \right)^{1/2}.
\end{equation}
We point out that the estimate (\ref{7.0-2}) will be proved in Section 8
with no smoothness assumption on $A$ (see Theorem \ref{theorem-4.1}).
Hence, for any $y\in B(0,1)$,
\begin{equation}\label{7.0-3}
\aligned
\left(\average_{B(y, \varep/2)} |\nabla u_\varep|^p \right)^{1/p}
 &\le C \left(\average_{B(y, 1)} |\nabla u_\varep|^2 \right)^{1/2}\\
 &\le C\, \| \nabla u_\varep\|_{L^2(B(0,2))}.
 \endaligned
\end{equation}
By covering $B(0,1)$ with balls of radius $\varep/2$,
we may deduce (\ref{7.0-0}) readily from (\ref{7.0-3}).
\end{proof}

\begin{lemma}\label{lemma-7.3}
Let $u_\varep\in H^1(D_{2r}; \rd)$ be a weak solution to $\mathcal{L}_\varep (u_\varep)=0$
in $D_{2r}$ with either $u_\varep=0$  or $\frac{\partial u_\varep}{\partial \nu_\varep}=0$
in $\Delta_{2r}$, where $0<r\le 1$.
Then, for any $2<p<\infty$,
\begin{equation}\label{7.3-0}
\left(\average_{D_r} |\nabla u_\varep|^p\right)^{1/p}
\le C_p \left(\average_{D_{2r}} |\nabla u_\varep|^2\right)^{1/2},
\end{equation}
where $C$ depends only on $d$,  $p$, $\kappa_1$, $\kappa_2$, $\tau(t)$ 
in (\ref{tau}), and $\rho(t)$ in (\ref{VMO}).
\end{lemma}

\begin{proof}
Note that the function $r^{-1}\psi (rx^\prime)$  satisfies the condition
(\ref{tau}) uniformly for $0<r\le 1$. Thus, by rescaling, it suffices to prove the lemma for $r=1$.
Using Lemma \ref{lemma-7.0}, Theorem \ref{theorem-6.1} and Theorem \ref{theorem-6.2},
we obtain 
\begin{equation}\label{7.3-2}
\aligned
\left(\average_{B(y, \delta (y)/8)} |\nabla u_\varep|^p\right)^{1/p}
&\le C \left(\average_{B(y, \delta (y)/4 )} |\nabla u_\varep|^2\right)^{1/2}\\
&\le C_\alpha \big[\delta(y) \big]^{\alpha -1} \|\nabla u_\varep\|_{L^2(D_2)},
\endaligned
\end{equation}
for any $\alpha\in (0,1)$,
where $y\in D_1$ and $\delta (y)=\text{dist}(y, \partial D_2)$.
We now fix $\alpha\in (1-\frac{1}{p}, 1)$.
It follows from (\ref{7.3-2}) that
\begin{equation}\label{7.3-3}
\int_{D_1} \left(\average_{B(y, \delta (y)/8)} |\nabla u_\varep|^p\, dx \right)\, dy
\le C\,  \|\nabla u_\varep \|^p_{L^2(D_2)}.
\end{equation}
Using the fact that $\delta(x)\approx \delta (y)$ if $y\in D_1$ and
$|y-x|<\frac{\delta(y)}{8}$,
it is not hard to verify that (\ref{7.3-3}) implies (\ref{7.3-0}).
\end{proof}

\begin{proof}[\bf Proof of Theorems \ref{theorem-7.1} and \ref{theorem-7.2}]
By duality and a density argument it suffices to consider the case where
$p>2$ and $f=(f_i^\alpha) \in C_0^1(\Omega; \mathbb{R}^{d\times d})$.
Furthermore, by a real-variable argument, which originated in \cite{ CP-1998} 
and further developed in \cite{Shen-2005, Shen-2007},
one only needs to establish weak reverse H\"older inequalities for
solutions of $\mathcal{L}_\varep (u_\varep)=0$ in $B(x_0, r)\cap \Omega$
with either $u_\varep=0$ or $\frac{\partial u_\varep}{\partial \nu_\varep}=0$
on $B(x_0, r)\cap \partial\Omega$, where $x_0\in \overline{\Omega}$
and $0<r<c_0\text{diam}(\Omega)$.
These inequalities are exactly those given by Lemmas \ref{lemma-7.0}
and \ref{lemma-7.3}. 
We omit the details and refer the reader to \cite{Shen-2005, Shen-2008, Geng-2012} for details 
 in the case of scalar elliptic equations.
\end{proof}

\begin{remark}\label{remark-7.1}
{\rm Suppose that $A$ and $\Omega$ satisfy the same conditions as in Theorem \ref{theorem-7.1}.
By some fairly standard extension and duality arguments  (see e.g. \cite{KLS1}),
one may deduce from Theorem \ref{theorem-7.1} that the solution of the Dirichlet 
problem,
$$
\mathcal{L}_\varep (u_\varep) =\text{div} (h) +F \quad \text{ in } \Omega
\quad \text{ and } \quad u_\varep =f \quad \text{ on } \partial\Omega,
$$
satisfies
$$
\|u_\varep\|_{W^{1, p}(\Omega)}
\le C_p \, \Big\{ \| h\|_{L^p(\Omega)} +\| F\|_{L^p(\Omega)} +\| f\|_{B^{\frac{1}{p}, p}(\partial\Omega)} \Big\},
$$
for any $1<p<\infty$, where $B^{\alpha, p}(\partial\Omega)$ denotes the Besov space on $\partial\Omega$
of order $\alpha$ with exponent $p$.
Similarly, the solutions of the Neumann problem,
$$
\mathcal{L}_\varep (u_\varep) =\text{div} (h) +F \quad \text{ in } \Omega
\quad \text{ and } \quad \frac{\partial u_\varep}{\partial \nu_\varep}
 =-n\cdot h +g  \quad \text{ on } \partial\Omega,
$$
with $u_\varep \perp \mathcal{R}$,
satisfies
$$
\|u_\varep\|_{W^{1, p}(\Omega)}
\le C_p \, \Big\{ \| h\|_{L^p(\Omega)} +\| F\|_{L^p(\Omega)} +\| g\|_{B^{-\frac{1}{p}, p}(\partial\Omega)} \Big\},
$$
where $B^{-1/p, p}(\partial\Omega)$ is the dual of $B^{1/p, p^\prime}(\partial\Omega)$ 
}
\end{remark}



\section{$L^p$ estimates in $C^1$ domains}
\setcounter{equation}{0}

The $W^{1, p}$ estimates in the last section allow us to establish the Rellich type estimates in  $L^p$, 
down to the scale $\varep$, in $C^1$ domains under the additional assumption that $A$ belongs to $V\!M\!O(\rd)$.

\begin{theorem}\label{theorem-8.1}
Suppose that $A=A(y)$ satisfies (\ref{ellipticity})-(\ref{periodicity}).
Also assume that $A\in V\!M\!O(\rd)$.
Let $1<p<\infty$ and $\Omega$ be a bounded $C^1$ domain in $\rd$.
Let $u_\varep\in W^{1,p}(\Omega; \rd)$ be a weak solution to the Dirichlet problem
\begin{equation} \label{DP-8}
\mathcal{L}_\varep (u_\varep) =F \quad \text{ in } \Omega \quad \text{ and } \quad
u_\varep =f \quad \text{ in } \partial\Omega,
\end{equation}
where $F\in L^p(\Omega; \rd)$ and $f\in W^{1, p}(\partial\Omega; \rd)$.
Then, for any $\varep\le r <\text{diam}(\Omega)$,
\begin{equation}\label{8.1-0}
\left\{ \frac{1}{r} \int_{\Omega_r} |\nabla u_\varep|^p \right\}^{1/p}
\le C_p \, \Big\{ \| F\|_{L^p(\Omega)} +\| f\|_{W^{1, p}(\partial\Omega)} \Big\},
\end{equation}
where $\Omega_r =\big\{ x\in \rd: \text{dist} (x, \partial\Omega) <r \big\}$.
The constant $C_p$ depends only on $d$, $p$, $A$ and $\Omega$.
\end{theorem}

\begin{theorem}\label{theorem-8.2}
Suppose that $A$ and $\Omega$ satisfy the same conditions as in Theorem \ref{theorem-8.1}.
Let $1<p<\infty$.
Let $u_\varep\in W^{1,p}(\Omega; \rd)$ be a weak solution to the Neumann problem
\begin{equation} \label{NP-8}
\mathcal{L}_\varep (u_\varep) =F \quad \text{ in } \Omega \quad \text{ and } \quad
\frac{\partial u_\varep}{\partial \nu_\varep} =g \quad \text{ in } \partial\Omega,
\end{equation}
where $F\in L^p(\Omega; \rd)$ , $g\in L^p (\partial\Omega; \rd)$ and $\int_\Omega F+\int_{\partial\Omega} g=0$.
Also assume that $u_\varep \perp \mathcal{R}$.
Then, for any $\varep\le r <\text{diam}(\Omega)$,
\begin{equation}\label{8.2-0}
\left\{ \frac{1}{r} \int_{\Omega_r} |\nabla u_\varep|^p \right\}^{1/p}
\le C_p \, \Big\{ \| F\|_{L^p(\Omega)} +\| g\|_{L^p (\partial\Omega)} \Big\},
\end{equation}
where  $C_p$ depends only on $d$, $p$, $A$ and $\Omega$.
\end{theorem}

The proof of Theorems \ref{theorem-8.1} and \ref{theorem-8.2} follows a
similar line of argument as for Theorems \ref{main-theorem-1} and \ref{main-theorem-2},
by considering 
\begin{equation}\label{w-8}
w_\varep = u_\varep -u_0 -\varep \chi_j^\beta (x/\varep) K_\varep \big(\frac{\partial u_0^\beta}{\partial x_j}
\eta_\varep\big),
\end{equation}
where $u_0$ is the solution of the homogenized problem, $K_\varep$ is a smoothing operator
defined by (\ref{K}), and $\eta_\varep\in C_0^\infty(\Omega)$ is a cut-off function  satisfying (\ref{eta}).

Throughout this section we will assume that $\Omega$ is $C^1$ and 
$A$ satisfies (\ref{ellipticity})-(\ref{periodicity}) and (\ref{VMO}).

\begin{lemma}\label{lemma-8.1}
Let $u_\varep$ $(\varep\ge 0)$ be the solutions of the Dirichlet problems (\ref{DP-8}).
Let $w_\varep$ be defined by (\ref{w-8}).
Then
\begin{equation}\label{8.3-0}
\| w_\varep \|_{W^{1,p} (\Omega)}
\le C_p\, \varep^{1/p}  \Big\{ \| f\|_{W^{1, p}(\partial\Omega)} +\| F\|_{L^p(\Omega)} \Big\},
\end{equation}
where $C_p$ depends only on $d$, $p$, $A$ and $\Omega$.
\end{lemma}

\begin{proof}
A direct computation shows that
$$
\aligned
\mathcal{L}_\varep (w_\varep)
= & -\frac{\partial}{\partial x_i}
\left\{ \Big[ \widehat{a}_{ij}^{\alpha\beta} -a_{ij}^{\alpha\beta} (x/\varep) \Big]
\Big[ \frac{\partial u_0^\beta}{\partial x_j} -K_\varep \Big(\frac{\partial u_0^\beta}{\partial x_j}
\eta_\varep \Big)\Big] \right\}\\
&+\frac{\partial}{\partial x_i} \left\{ b_{ij}^{\alpha\beta} (x/\varep) K_\varep \Big(\frac{\partial u_0^\beta}
{\partial x_j} \eta_\varep \Big) \right\}\\
& +\varep \frac{\partial}{\partial x_i}
\left\{ a_{ij}^{\alpha\beta}(x/\varep) \chi_k^{\beta\gamma} (x/\varep)
\frac{\partial}{\partial x_j} \Big(K_\varep \Big(\frac{\partial u_0^\gamma}
{\partial x_k} \eta_\varep\Big)\Big)\right\},
\endaligned
$$
where $b_{ij}^{\alpha\beta}(y)$ is defined by (\ref{b}).
Using (\ref{phi}), we obtain 
$$
\frac{\partial}{\partial x_i} \left\{ b_{ij}^{\alpha\beta} (x/\varep) K_\varep \Big(\frac{\partial u_0^\beta}
{\partial x_j} \eta_\varep \Big) \right\}
=-\varep \frac{\partial}{\partial x_i}
\left\{ \phi^{\alpha\beta}_{kij}(x/\varep) \frac{\partial}{\partial x_k} \Big( K_\varep \Big( \frac{\partial u_0^\beta}{\partial x_j}
\eta_\varep\Big)\Big) \right\}.
$$
It follows that
\begin{equation}\label{key-f}
\aligned
\mathcal{L}_\varep (w_\varep)
= & -\frac{\partial}{\partial x_i}
\left\{ \Big[ \widehat{a}_{ij}^{\alpha\beta} -a_{ij}^{\alpha\beta} (x/\varep) \Big]
\Big[ \frac{\partial u_0^\beta}{\partial x_j} -K_\varep \Big(\frac{\partial u_0^\beta}{\partial x_j}
\eta_\varep \Big)\Big] \right\}\\
&
-\varep \frac{\partial}{\partial x_i}
\left\{ \phi^{\alpha\beta}_{kij}(x/\varep) \frac{\partial}{\partial x_k} \Big( K_\varep \Big( \frac{\partial u_0^\beta}{\partial x_j}
\eta_\varep\Big)\Big) \right\}\\
& +\varep \frac{\partial}{\partial x_i}
\left\{ a_{ij}^{\alpha\beta}(x/\varep) \chi_k^{\beta\gamma} (x/\varep)\frac{\partial}{\partial x_j}
\Big(K_\varep \Big(\frac{\partial u_0^\gamma}
{\partial x_k} \eta_\varep\Big)\Big)\right\}.
\endaligned
\end{equation}
Since $w_\varep=0$ on $\partial\Omega$,
we may apply the $W^{1,p}$ estimate in Theorem \ref{theorem-7.1} to obtain 
\begin{equation}\label{8.3-2}
\aligned
\| w_\varep\|_{W^{1, p}(\Omega)}
\le  &C \Big\{ \|\nabla u_0 - K_\varep \big( (\nabla u_0)\eta_\varep \big) \|_{L^p(\Omega)}
 +\varep \|\phi(x/\varep) \nabla K_\varep \big((\nabla u_0)\eta_\varep\big) \|_{L^p(\Omega)}\\
& \qquad\qquad +\varep \| \chi(x/\varep) \nabla K_\varep \big( (\nabla u_0)\eta_\varep \big) \|_{L^p(\Omega)} \Big\}\\
& \le C \Big\{ \|\nabla u_0 - K_\varep \big( (\nabla u_0)\eta_\varep \big) \|_{L^p(\Omega)}
 +\varep \|\nabla  \big((\nabla u_0)\eta_\varep\big) \|_{L^p(\Omega)} \Big\}\\
 &\le C \Big\{ \|\nabla u_0\|_{L^p(\Omega_{4\varep})}
 +\varep\, \|(\nabla^2 u_0)\eta_\varep\|_{L^2(\Omega)} \Big\},
\endaligned
\end{equation}
where we have used Lemma \ref{lemma-2.00} for the second and third inequalities.

We now write $u_0=v +w$, where
\begin{equation}\label{f-rep}
v(x) =\int_\Omega \Gamma_0 (x-y) F(y)\, dy
\end{equation}
and $\Gamma_0 (x-y)$ denotes the matrix of fundamental solutions for the operator $\mathcal{L}_0$
in $\rd$, with pole at the origin.
Note that $\| v\|_{W^{2, p}(\rd)} \le C_p\,  \| F\|_{L^p(\Omega)}$ and
$$
\| \nabla v\|_{L^p(S_t)} \le C_p\,  \| F\|_{L^p(\Omega)},
$$
where $S_t =\{ x\in \rd: \text{dist}(x, \partial\Omega) =t\}$ for $t$ small
(see the proof of  Theorem \ref{theorem-2.1}). It follows that
\begin{equation}\label{8.3-3}
\|\nabla v\|_{L^p(\Omega_{4\varep})}
+\varep \|\nabla^2 v \|_{L^p(\Omega)} \le C\, \varep^{1/p} \| F\|_{L^p(\Omega)}.
\end{equation}

Finally, we observe that $\mathcal{L}_0 (w)=0$ in $\Omega$ and 
$$
\aligned
\| w\|_{W^{1, p}(\partial\Omega)} & \le \| f\|_{W^{1,p}(\partial\Omega)}
+ \| v\|_{W^{1, p}(\partial\Omega)} \\
&\le C \Big\{ \| f\|_{W^{1, p}(\partial\Omega)}
+\| F\|_{L^p(\Omega)} \Big \}.
\endaligned
$$
It follows from the solvability of the $L^p$ regularity problem for the operator $\mathcal{L}_0$ in
$C^1$ domain $\Omega$, which follows from  \cite{Fabes-1978, Lewis-1993, Hofmann-C-1}, that
$$
\| (\nabla w)^*\|_{L^p(\partial\Omega)}
\le C\,  \Big\{ \| f\|_{W^{1, p}(\partial\Omega)}
+\| F\|_{L^p(\Omega)} \Big \}.
$$
Also, using the interior estimate
$$
|\nabla^2 w (x)|\le \frac{C}{\delta (x)} \left(\average_{B(x, \delta (x)/8)} |\nabla w|^p\right)^{1/p},
$$
where $\delta(x)=\text{dist}(x, \partial\Omega)$, we may show that
$$
\aligned
\int_{\Omega\setminus\Omega_{3\varep}} |\nabla^2 w|^p \, dx
& \le C \int_{\Omega\setminus \Omega_{2\varep}}
|\nabla w(x)|^p \big[\delta(x)\big]^{-p}\, dx \\
&\le C\varep^{1-p} \|(\nabla w)^*\|_{L^p(\partial\Omega)}^p\\
&\le C\varep^{1-p} 
\Big\{ \| f\|^p_{W^{1, p}(\partial\Omega)}
+\| F\|^p_{L^p(\Omega)} \Big \}.
\endaligned
$$
As a result, we obtain 
$$
\| \nabla w\|_{L^p(\Omega_{4\varep})} 
+\varep \|(\nabla^2 w)\eta_\varep\|_{L^p(\Omega)}
\le C\, \varep^{1/p} 
\Big\{ \| f\|_{W^{1, p}(\partial\Omega)}
+\| F\|_{L^p(\Omega)} \Big \}.
$$
This, together with the estimate (\ref{8.3-3}) for $v$, gives
\begin{equation}\label{8.3-7}
\| \nabla u_0\|_{L^p(\Omega_{4\varep})} 
+\varep \|(\nabla^2 u_0)\eta_\varep\|_{L^p(\Omega)}
\le C\, \varep^{1/p} 
\Big\{ \| f\|_{W^{1, p}(\partial\Omega)}
+\| F\|_{L^p(\Omega)} \Big \},
\end{equation}
which, in view of (\ref{8.3-2}), completes the proof.
\end{proof}

\begin{proof}[\bf Proof of Theorem \ref{theorem-8.1}]
Without loss of generality we may assume that
$$
\| f\|_{W^{1,p}(\partial\Omega)} +\| F\|_{L^p(\Omega)} =1.
$$
Let $\varep\le r <\text{diam}(\Omega)$.
It follows from Lemma \ref{lemma-8.1} that
\begin{equation}\label{8.4-1}
\aligned
\|\nabla u_\varep \|_{L^p(\Omega_r)}
&\le \|\nabla u_0\|_{L^p(\Omega_r)}
+C  \| \nabla \chi (x/\varep) K_\varep \big( (\nabla u_0)\eta_\varep \big) \|_{L^p(\Omega_r)}\\
& \qquad\qquad
+C\varep \|\chi(x/\varep) \nabla K_\varep \big( (\nabla u_0)\eta_\varep\big) \|_{L^p(\Omega_r)}
+C \varep^{1/p}\\
&\le  C\,  \| \nabla u_0\|_{L^p(\Omega_{2r})} 
+ C \varep \|\nabla \big( (\nabla u_0)\eta_\varep \big)\|_{L^p(\Omega)}
+C \varep^{1/p}\\
& \le C\,  \| \nabla u_0\|_{L^p(\Omega_{2r})}  + C \varep^{1/p},
\endaligned
\end{equation}
where we have used Lemma \ref{lemma-2.00} for the second inequality and (\ref{8.3-7}) for the third.
An inspection of the proof of Lemma \ref{lemma-8.1} shows that
$$
\|\nabla u_0\|_{L^p(\Omega_{2r})} \le C \, r^{1/p},
$$
which, in view of (\ref{8.4-1}), gives 
$$
\|\nabla u_\varep\|_{L^p(\Omega_r)} \le C \, r^{1/p}.
$$
This completes the proof.
\end{proof}

To prove Theorem \ref{theorem-8.2}, we need the following lemma.

\begin{lemma}\label{lemma-8.2}
Let $u_\varep$ $(\varep\ge 0)$ be solutions of the Neumann problem (\ref{NP-8}).
Also assume that $u_\varep, u_0 \perp \mathcal{R}$.
Let $w_\varep$ be defined by (\ref{w-8}).
Then 
\begin{equation}\label{8.5-0}
\| w_\varep\|_{W^{1, p}(\Omega)} 
\le C_p\,  \varep^{1/p} \Big\{ \| g\|_{L^p(\partial\Omega)} +\| F\|_{L^p(\Omega)} \Big\},
\end{equation}
where $C_p$ depends only on $d$, $p$, $A$ and $\Omega$.
\end{lemma}

\begin{proof}
The proof is similar to that of Lemma \ref{lemma-8.1}.
Let $\phi_\varep$ be a function in $\mathcal{R}$ such that
$w_\varep -\phi_\varep \perp \mathcal{R}$ in $L^2(\Omega; \rd)$.
It follows from the formula (\ref{key-f}) and the $W^{1, p}$ estimates in Theorem \ref{theorem-7.2}
that
\begin{equation}\label{8.5-1}
\| w_\varep -\phi_\varep\|_{W^{1,p}(\Omega)}
\le C \Big\{ \|\nabla u_0\|_{L^p(\Omega_{4\varep})}
+ \varep \| (\nabla^2 u_0)\eta_\varep\|_{L^2(\Omega)} \Big\}.
\end{equation}
To estimate the right hand side of (\ref{8.5-1}),
we proceed as in the proof of Lemma \ref{lemma-8.1}, but use the nontangential maximal function 
estimate \cite{Fabes-1978, Lewis-1993, Hofmann-C-1},
$$
\| (\nabla w)^*\|_{L^p(\partial\Omega)} \le C\, \|\frac{\partial w}{\partial \nu_0} \|_{L^p(\partial\Omega)},
$$
where $\mathcal{L}_0 (w)=0$ in $\Omega$ and $ w \perp \mathcal{R}$ in $L^2(\Omega; \rd)$.
As a result, we obtain 
\begin{equation}\label{8.5-2}
\| w_\varep -\phi_\varep\|_{W^{1, p}(\Omega)}
\le C\, \varep^{1/p} \Big\{ \| g\|_{L^p(\partial\Omega)} +\| F\|_{L^p(\Omega)} \Big\}.
\end{equation}
Finally, note that since $u_\varep -u_0 \perp \mathcal{R}$, 
$$
\aligned
\|\phi_\varep \|_{W^{1, p}(\Omega)}
 & \le C\varep \| \chi(x/\varep) K_\varep \big( (\nabla u_0)\eta_\varep\big) \|_{L^p(\Omega)}\\
 &\le C\varep \|\nabla u_0\|_{L^p(\Omega)}.
 \endaligned
 $$
 This, together with (\ref{8.5-2}), yields the estimate (\ref{8.5-0}).
\end{proof}

\begin{proof}[\bf Proof of Theorem \ref{theorem-8.2}]
The estimate (\ref{8.2-0}) follows from (\ref{8.5-0}), as in the case of the Dirichlet conditions.
We omit the details.
\end{proof}

\begin{remark}\label{remark-8.1}
{\rm
Under certain  smoothness condition on $A$, such as H\"older continuity,
it is possible to solve the $L^p$ Dirichlet,  regularity, and Neumann problems for 
$\mathcal{L}_1 (u)=0$ in $C^1$ domains for any $1<p<\infty$.
By the same  localization procedure and blow-up argument as in Remark \ref{remark-3.1},
this  implies that
\begin{equation}\label{re-8.1-1}
\left\{
\aligned
&\int_{\partial \Omega} |\nabla u_\varep|^p\, d\sigma
\le C \int_{\partial\Omega} \Big|\frac{\partial u_\varep}{\partial \nu_\varep} \Big|^p\, d\sigma
+\frac{C}{\varep} \int_{\Omega_{c\varep } }  |\nabla u_\varep|^p\, dx,\\
&\int_{\partial \Omega} |\nabla u_\varep|^p\, d\sigma
\le C \int_{\partial\Omega} \Big|\nabla_{\tan} u_\varep \Big|^p\, d\sigma
+\frac{C}{\varep} \int_{\Omega_{c\varep}} |\nabla u_\varep|^p\, dx,
\endaligned
\right.
\end{equation}
where $\mathcal{L}_\varep (u_\varep)=0$ in $\Omega$.
It then follows from Theorems \ref{theorem-8.1} and \ref{theorem-8.2} that
\begin{equation}\label{re-8.1-2}
\int_{\partial\Omega} |\nabla u_\varep|^p\, d\sigma
\le C \int_{\partial\Omega} \Big|\frac{\partial u_\varep}{\partial \nu_\varep} \Big|^p\, d\sigma,
\end{equation}
if $u_\varep \perp \mathcal{R}$, and
\begin{equation}\label{re-8.1-3}
\int_{\partial\Omega} |\nabla u_\varep|^p\, d\sigma
\le C  \int_{\partial\Omega} |\nabla_{\tan} u_\varep|^p\, d\sigma
+C\int_{\partial\Omega} |u_\varep|^p\, d\sigma.
\end{equation}
As in the case $p=2$,
by the method of layer potentials, estimates (\ref{re-8.1-2})-(\ref{re-8.1-3})
lead to the uniform solvability of the $L^p$ Dirichlet, regularity, and Neumann problems
in $C^1$ domains.
The details will be given elsewhere.
}
\end{remark}



\section{Lipschitz estimates in $C^{1,\alpha}$ domains, part I}
\setcounter{equation}{0}

In this section we investigate the Lipschitz estimates, down to the scale $\varep$,
 in $C^{1, \alpha}$ domains
with Dirichlet boundary conditions and give the proof of
Theorem \ref{main-theorem-3}. The Neumann boundary conditions will be treated in the next section.
The proof of Theorems \ref{main-theorem-3} and \ref{main-theorem-4}
is based on a general scheme for establishing Lipschitz estimates at large scales 
in homogenization, recently formulated in \cite{Armstrong-Smart-2014}
for interior estimates. Our approach to the boundary Lipschitz estimates in $C^{1, \alpha}$
domains is similar to that used in \cite{Armstrong-Shen-2014} for elliptic systems with
almost-periodic coefficients.
We remark that Lemma \ref{G-lemma} is a continuous version of  
Lemma 3.1 in \cite{Armstrong-Shen-2014}.

Let $D_r$ and $\Delta_r$ be defined by (\ref{D}) with $\psi (0)=0$ and
$\|\nabla \psi\|_\infty \le M$.

\begin{lemma}\label{lemma-4.2} 
Let $u_\varep \in H^1(D_2; \rd)$ be a weak solution of
$\mathcal{L}_\varep (u_\varep)=F$ in $D_2$ with $u_\varep =f$ on $\Delta_2$.
Then there exists $v\in H^1(D_1; \rd)$ such that $\mathcal{L}_0 (v)=F$ in $D_1$,
$v= f$ on $\Delta_1$, and
\begin{equation}\label{4.2-0}
\| u_\varep -v\|_{L^2(D_1)}
  \le C\varep^{1/2}
\Big\{\| u_\varep\|_{L^2(D_2)} 
+ \| F\|_{L^2(D_2)}
+\|  f\|_{L^\infty(\Delta_2)} +\| \nabla_{\tan} f\|_{L^\infty(\Delta_2)} \Big\},
\end{equation}
where $C$ depends only on $d$, $\kappa_1$, $\kappa_2$, and $M$.
\end{lemma}

\begin{proof}
By Cacciopoli's inequality,
$$
\int_{D_{3/2}} |\nabla u_\varep|^2 \le C \left\{ \int_{D_2} |u_\varep|^2 +\int_{D_2} |F|^2
+\| f\|_{L^\infty(\Delta_2)}^2 +\|\nabla_{\tan} f\|^2_{L^\infty(\Delta_2)} \right\}.
$$
By the co-area formula this implies that  there exists some $t\in [5/4,3/2]$ such that
$$
\int_{\partial D_t\setminus \Delta_2}\Big( |\nabla u_\varep|^2+ |u_\varep|^2 \Big)
 \le C \left\{ \int_{D_2} |u_\varep|^2 +\int_{D_2} |F|^2
+\| f\|_{L^\infty(\Delta_2)}^2 +\|\nabla_{\tan} f\|^2_{L^\infty(\Delta_2)} \right\}.
$$
Let $v$ be the weak solution to the Dirichlet problem,
$$
\mathcal{L}_0 (v)=F \quad \text{ in } D_t \quad \text{ and } \quad v=u_\varep \quad \text{ on } \partial D_t.
$$
It follows from Remark \ref{remark-2.1} that
$$
\aligned
\| u_\varep -v\|_{L^2(D_1)}
&\le \| u_\varep -v\|_{L^2(D_t)}\\
&\le C\varep^{1/2} \Big\{ \| u_\varep\|_{H^1(\partial D_t)} +\| F\|_{L^2(D_t)}\Big\}\\
& \le C \varep^{1/2}
\Big\{ \| u_\varep\|_{L^2(D_2)} +\| F\|_{L^2(D_2)} +\| f\|_{L^\infty(\Delta_2)}
+\|\nabla_{\tan} f\|_{L^\infty(\Delta_2)} \Big\},
\endaligned
$$
where $C$ depends only on $d$, $\kappa_1$, $\kappa_2$, and $M$.
\end{proof}

\begin{lemma}\label{lemma-4.3} 
Let $\varep\le r<1$.
Let $u_\varep \in H^1(D_{2r}; \rd)$ be a weak solution of
$\mathcal{L}_\varep (u_\varep)=F$ in $D_{2r}$ with $u_\varep =f$ on $\Delta_{2r}$.
Then there exists $v\in H^1(D_r; \rd)$ such that $\mathcal{L}_0 (v)=F$ in $D_r$,
$v= f$ on $\Delta_r$, and
\begin{equation}\label{4.3-0}
\aligned
\left(\average_{D_r} |u_\varep -v|^2\right)^{1/2}
   \le C\left(\frac{\varep}{r} \right)^{1/2}
& \Bigg\{ \left(\average_{D_{2r}} |u_\varep|^2 \right)^{1/2}
 + r^2 \left(\average_{D_{2r}} |F|^2\right)^{1/2}\\
&\qquad  +\|  f\|_{L^\infty(\Delta_{2r})} + r\| \nabla_{\tan} f\|_{L^\infty(\Delta_{2r})} \Bigg\},
\endaligned
\end{equation}
where $C$ depends only on $d$, $\kappa_1$, $\kappa_2$, and $M$.
\end{lemma}

\begin{proof}
This follows from Lemma \ref{lemma-4.2} by rescaling.
\end{proof}


In the rest of this section we will assume that the defining function $\psi$ in the definition of
$D_r$ and $\Delta_r$ is $C^{1, \alpha}$ for some $\alpha \in (0,1)$ with $\psi(0)=0$ and
$
\|\nabla \psi \|_{C^{\alpha}(\mathbb{R}^{d-1})} \le M.
$

\begin{lemma}\label{lemma-4.4}
Let $v$ be a solution of $\mathcal{L}_0 (v)=F$ in $D_r$ with $v=f$ on $\Delta_r$.
For $0<t\le r$, define
\begin{equation}\label{G}
\aligned
G(t; v)=  \frac{1}{t} \inf_{\substack{ M\in \mathbb{R}^{d\times d}\\ q\in \rd}} &
\Bigg\{  \left(\average_{D_t} |v- Mx -q|^2 \right)^{1/2}
+ t^2\left(\average_{D_t} |F|^p\right)^{1/p} \\ 
&+\| f- Mx-q\|_{L^\infty(\Delta_t)}
+ t\, \|\nabla_{\tan} (f-Mx-q) \|_{L^\infty(\Delta_t)}\\
&+t^{1+\sigma}\,  \|\nabla_{\tan} (f -Mx-q)\|_{C^{0, \sigma} (\Delta_t)} \Bigg\},
\endaligned
\end{equation}
where $p>d$ and $\sigma \in (0, \alpha)$.
Then there exists $\theta\in (0,1/4)$, depending only on $d$, $p$, $\kappa_1$,
$\kappa_2$, $\sigma$, $\alpha$ and $M$, such that
\begin{equation}\label{4.4-0}
G(\theta r; v) \le (1/2) G(r; v).
\end{equation}
\end{lemma}

\begin{proof}
The lemma follows from the boundary $C^{1, \alpha}$ estimates for elasticity systems with constant coefficients.
We refer the reader to \cite[Lemma 7.1]{Armstrong-Shen-2014}  for the case $\mathcal{L}_0 (v)=0$. 
The argument for the general case $F\in L^p$ with $p>d$ is the same.
\end{proof}

\begin{lemma}\label{lemma-4.5}
Let $0<\varep<1/2$.
Let $u_\varep$ be a solution of $\mathcal{L}_\varep (u_\varep)=F$ in $D_1$ with $u_\varep=f$ on $\Delta_1$.
Define
\begin{equation}\label{H}
\aligned
H(r)=  \frac{1}{r} \inf_{\substack{ M\in \mathbb{R}^{d\times d}\\ q\in \rd}} &
\Bigg\{  \left(\average_{D_r} |u_\varep- Mx -q|^2 \right)^{1/2}
+ r^2\left(\average_{D_r} |F|^p\right)^{1/p} \\ 
&+\| f- Mx-q\|_{L^\infty(\Delta_r)}
+ r\, \|\nabla_{\tan} (f-Mx-q) \|_{L^\infty(\Delta_r)}\\
&+r^{1+\sigma}\,  \|\nabla_{\tan} (f -Mx-q)\|_{C^{0, \sigma} (\Delta_r)} \Bigg\},
\endaligned
\end{equation}
and
\begin{equation}\label{Phi}
\aligned
\Phi (r)
=\inf_{q\in \rd}
 \Bigg\{  &\left(\average_{D_{2r}}   |u_\varep -q |^2 \right)^{1/2}
  + r^2 \left(\average_{D_{2r}} |F|^p\right)^{1/p}\\
&  \qquad\qquad
+\|  f-q \|_{L^\infty(\Delta_{2r})} + r\| \nabla_{\tan} f\|_{L^\infty(\Delta_{2r})} \Bigg\},
  \endaligned
  \end{equation}
where $p>d$ and $\sigma \in (0, \alpha)$.
Then
\begin{equation}\label{4.5-0}
H(\theta r) \le (1/2) H(r) +C \left(\frac{\varep}{r}\right)^{1/2} \Phi (2r),
\end{equation}
for any $r  \in [\varep,1/2]$, where $\theta\in (0,1/4)$ is given by Lemma \ref{lemma-4.4}.
\end{lemma}

\begin{proof}
Fix $r\in [\varep, 1/2]$.
Let $v$ be a solution of $\mathcal{L}_0 (v)=F$ in $D_r$ with $v =f$ on $\Delta_r$.
Observe that
$$
\aligned
H(\theta r)
& \le \left(\average_{D_{\theta r}} |u_\varep -v|^2\right)^{1/2}
+ G(\theta r; v)\\
&\le \left(\average_{D_{\theta r}} |u_\varep -v|^2\right)^{1/2}
+ (1/2) G(r; v)\\
&\le C \left(\average_{D_{r}} |u_\varep -v|^2\right)^{1/2} + (1/2) H(r),
\endaligned
$$
where we have used Lemma \ref{lemma-4.4} for the second inequality.
This, together with Lemma \ref{lemma-4.3}, gives
$$
\aligned
H(\theta r) \le (1/2) H(r) +
 C\left(\frac{\varep}{r} \right)^{1/2}
& \Bigg\{ \left(\average_{D_{2r}} |u_\varep|^2 \right)^{1/2}
 + r^2 \left(\average_{D_{2r}} |F|^2\right)^{1/2}\\
&\qquad  +\|  f\|_{L^\infty(\Delta_{2r})} + r\| \nabla_{\tan} f\|_{L^\infty(\Delta_{2r})} \Bigg\}.
\endaligned
$$
Since $H(r)$ remains invariant if we subtract a constant from $u_\varep$, the inequality (\ref{4.5-0})
follows.
\end{proof}

\begin{lemma}\label{G-lemma}
Let $H(r)$ and $h(r)$ be two nonnegative  continuous functions on the interval $(0, 1]$. Let
$0<\varep<(1/4)$.
Suppose that there exists a constant $C_0$  such that
\begin{equation}\label{G-1}
\left\{
\aligned
& \max_{r\le t \le 2r} H(t) \le C_0\,  H(2r),\\
& \max_{r\le t,s\le 2r} |h(t) -h(s)|   \le C_0\,  H(2r),
\endaligned
\right.
\end{equation}
for any $r\in [\varep, 1/2]$. We further assume that
\begin{equation}\label{G-3}
H(\theta r) \le (1/2) H(r) + C_0\,  \omega (\varep/r) \Big\{ H(2r) + h(2r) \Big\},
\end{equation}
for any $r\in [ \varep, 1/2]$,
where $\theta\in (0,1/4)$ and $\omega$ is a nonnegative increasing function $[0,1]$ such that
$\omega(0)=0$ and
\begin{equation}\label{G-4}
\int_0^1 \frac{\omega(t)}{t}\, dt  <\infty.
\end{equation}
Then
\begin{equation}\label{G-5}
\max_{\varep\le r\le 1}
\Big\{ H(r) +h(r) \Big\}
\le C \Big\{ H(1) +h (1) \Big\},
\end{equation}
where $C$ depends only on $C_0$, $\theta$,  and $\omega$.
\end{lemma}

\begin{proof}
It follows from (\ref{G-1}) that
$$
h(r) \le h(2r) + C_0\, H(2r)
$$
for any $\varep\le r\le 1/2$.
Hence,
$$
\aligned
\int_a^{1/2} \frac{h(r)}{r} \, dr  &\le 
\int_a^{1/2} \frac{h(2r)}{r} \, dr + C_0 \int_a^{1/2} \frac{H(2r)}{r} \, dr \\
&= \int_{2a}^{1} \frac{h(r)}{r} \, dr +
 C_0 \int_{2a}^{1} \frac{H(r)}{r} \, dr,
 \endaligned
 $$
where $\varep\le a\le (1/4)$.
This implies that
$$
\aligned
\int_a^{2a} \frac{h(r)}{r}\, dr
&\le \int_{1/2}^1 \frac{h(r)}{r}\, dr +C\int_{2a}^1 \frac{H(r)}{r}\, dr\\
&\le C \big\{ h(1) + H(1) \big\} +C\int_{2a}^1 \frac{H(r)}{r}\, dr,
\endaligned
$$
which, by (\ref{G-1}), gives
\begin{equation}\label{G-6}
\aligned
h(a) & \le C \left\{ H(2a) + h(1) + H(1) + \int_{2a}^1 \frac{H(r)}{r}\, dr\right\}\\
 & \le C \left\{ h(1) + H(1) + \int_{a}^1 \frac{H(r)}{r}\, dr\right\},
 \endaligned
\end{equation}
for any $a\in [\varep, 1/4]$.

Next, we use (\ref{G-3}) and (\ref{G-6}) to obtain
$$
H(\theta r) \le (1/2) H(r) +C\,  \omega (\varep/r) \big\{ h(1) +H(1) \big\}
+C \, \omega(\varep/r) 
\int_{r} ^1 \frac{H(r)}{r}\, dr.
$$
It follows that 
$$
\int_{\alpha \theta \varep }^{\theta }
\frac{H( r)}{r}\, dr
\le \frac12 \int_{\alpha\varep}^1 \frac{H( r)}{r}\, dr
 +C_\alpha \big\{ h(1) +H(1) \big\} 
 + C\,  \int_{\alpha\varep}^{1} \omega(\varep/r) \left\{  \int_r^1 \frac{H(t)}{t}\,dt\right\} \frac{dr}{r},
 $$
 where $\alpha>1$ and we have used the condition (\ref{G-4}).
 Using (\ref{G-4}) and the observation that 
 $$
 \aligned
 \int_{\alpha\varep}^{1} \omega(\varep/r) \left\{  \int_r^1 \frac{H(t)}{t}\,dt\right\} \frac{dr}{r}
  &=\int_{\alpha\varep}^1 H(t) \left\{ \int_{\frac{\varep}{t}}^{\frac{1}{\alpha}}
  \frac{\omega (s)}{s}\, ds \right\} \frac{dt}{t}\\
  &\le (4C)^{-1}
  \int_{\alpha\varep}^1 H(t) \frac{dt}{t}
  \endaligned
 $$
 if $\alpha>\alpha_0(\omega)$, we see that
 $$
 \int_{\alpha \theta \varep }^{\theta }
\frac{H( r)}{r}\, dr
\le \frac12 \int_{\alpha\varep}^1 \frac{H( r)}{r}\, dr
 +C_\alpha \big\{ h(1) +H(1) \big\} 
 + \frac14 \int_{\alpha\varep}^{1}\frac{H(r)}{r}\, dr.
 $$
It follows that
\begin{equation}\label{G-10}
\int_{\varep}^1 \frac{H(r)}{r}\, dr \le C \big\{ h(1) +H(1) \big\},
\end{equation}
which, together with (\ref{G-1}) and (\ref{G-6}),  yields the estimate (\ref{G-5}).
This completes the proof.
\end{proof}

\begin{proof}[\bf Proof of Theorem \ref{main-theorem-3}]
We may assume that $0<\varep<(1/4)$.
Let $u_\varep$ be a solution of $\mathcal{L}_\varep (u_\varep)=F$ in $D_1$ with 
$u_\varep =f$ on $\Delta_1$, where $F\in L^p(D_1)$ for some $p>d$ and
$f\in C^{1, \sigma}(\Delta_1)$ for some $\sigma\in (0, \alpha)$.
For  $r\in (0,1)$, we define the function $H(r)$ by (\ref{H}).
It is easy to see that $H(t)\le C\, H(2r)$ if $t\in (r, 2r)$.

Next, we let $h(r)=|M_r |$, where $M_r $ is the $d\times d$ matrix such that
$$
\aligned
H(r)=  \frac{1}{r} \inf_{q\in \rd} &
\Bigg\{  \left(\average_{D_r} |u_\varep- M_r x -q|^2 \right)^{1/2}
+ r^2\left(\average_{D_r} |F|^p\right)^{1/p} \\ 
&\qquad +\| f- M_r x-q\|_{L^\infty(\Delta_r)}
+ r\, \|\nabla_{\tan} (f-M_r x-q) \|_{L^\infty(\Delta_r)}\\
&\qquad +r^{1+\sigma}\,  \|\nabla_{\tan} (f -M_r x-q)\|_{C^{0, \sigma} (\Delta_r)} \Bigg\}.
\endaligned
$$
Let $t, s\in [r,2r]$. Using
$$
\aligned
|M_t-M_s|
 &\le  \frac{C}{r} \inf_{q\in \rd} \left(\average_{D_r} |(M_t-M_s)x -q|^2\right)^{1/2}\\
 &\le \frac{C}{t}\inf_{q\in \rd} \left(\average_{D_t} |u_\varep -M_t x -q|^2 \right)^{1/2}
 +\frac{C}{s}\inf_{q\in \rd} \left(\average_{D_s} |u_\varep - M_s x -q|^2 \right)^{1/2}\\
 &\le C \big\{ H(t) +H(s) \big\}\\
 &\le C H(2r),
 \endaligned
 $$
 we obtain 
 $$
 \max_{r\le t,s\le 2r} | h(t)-h(s)| \le C \, H(2r).
 $$
 Furthermore, if $\Phi$ is defined by (\ref{Phi}), then
 $$
 \Phi (r)\le H(2r) + h(2r).
 $$
 In view of Lemma \ref{lemma-4.5} this gives
 $$
 H(\theta r) \le (1/2) H(r) +C \omega (\varep/r)  \big\{ H(2r) + h(2r) \Big\}
 $$
 for $r\in [\varep, 1/2]$,
 where $\omega (t)=t^{1/2}$.
 Thus the functions $H(r)$ and $h(r)$ satisfy the conditions (\ref{G-1}), (\ref{G-3}) and
 (\ref{G-4}) in Lemma \ref{G-lemma}. Consequently, we obtain that for $r\in [\varep, 1/2]$,
 $$
 \aligned
 \inf_{q\in \rd} \frac{1}{r} \left(\average_{D_r} |u_\varep -q|^2 \right)^{1/2}
&\le C \Big\{ H(r) +h (r)\Big\}  \\
 & \le C \big\{ H(1) + h(1) \big\}\\
 & \le C \left\{ \left(\average_{D_1} |u_\varep|^2\right)^{1/2}
 +\|F\|_{L^p(D_1)} + \| f\|_{C^{1,\sigma}(\Delta_1)} \right\},
 \endaligned
 $$
 which, together with Cacciopli's inequality, gives the estimate (\ref{main-estimate-3}).
 The proof is complete.
\end{proof}

The argument used in this section may be used to prove the interior Lipschitz estimates,
down to the scale $\varep$.

\begin{theorem}\label{theorem-4.1}
Suppose that $A$ satisfies (\ref{ellipticity})-(\ref{periodicity}).
Let $u_\varep\in H^1(B(x_0, R); \rd)$ be a weak solution of $\mathcal{L}_\varep (u_\varep)=F$
in $B(x_0,R)$ for some $x_0\in \rd$ and $R>0$, where $F\in L^p(B(x_0,R); \rd)$ for some $p>d$. Then, for
$\varep\le r< R$,
\begin{equation}\label{i-Lip}
\left(\average_{B(x_0, r)} |\nabla u_\varep|^2 \right)^{1/2}
\le C \left\{ \left(\average_{B(x_0, R)} |\nabla u_\varep|^2 \right)^{1/2}
+ R\left(\average_{B(x_0, R)} |F|^p\right)^{1/p} \right\},
\end{equation}
where $C$ depends only on $d$, $\kappa_1$, $\kappa_2$, and $p$.
\end{theorem}



\section{Lipschitz estimates in $C^{1, \alpha}$ domains, part II}
\setcounter{equation}{0}

In this section we study the Lipschitz estimate, down to the scale $\varep$,
with Neumann boundary conditions, and 
give the proof of Theorem \ref{main-theorem-4}.
Throughout this section we will assume that the defining function $\psi$
in $D_r$ and $\Delta_r$ is $C^{1,\alpha}$ for some $\alpha\in (0,1)$ and
$\|\nabla \psi\|_{C^\alpha(\mathbb{R}^{d-1})}\le M$.

\begin{lemma}\label{lemma-5.1}
Let $\Omega$ be a bounded Lipschitz domain.
Let $u_\varep\in H^1(\Omega; \rd)$ be a weak solution to the Neumann problem:
$\mathcal{L}_\varep (u_\varep)=F$ in $\Omega$ and $\partial u_\varep/\partial\nu_\varep =g$ on $\partial\Omega$.
Then there exists $w \in H^1(\Omega; \rd)$ such that
$\mathcal{L}_0 (w)=F$ in $\Omega$, $\partial w/\partial\nu_0=g$ on $\partial\Omega$, and
\begin{equation}\label{5.1-0}
\| u_\varep -w\|_{L^2(\Omega)}
\le C \, \varep^{1/2} 
\Big\{ \| g\|_{L^2(\partial\Omega)} + \| F\|_{L^2(\Omega)} \Big\}.
\end{equation} 
\end{lemma}

\begin{proof}
Choose $\phi_\varep \in \mathcal{R}$ such that $u_\varep -\phi_\varep \perp \mathcal{R}$
in $L^2(\Omega; \rd)$.
Let $u_0$ be the weak solution to the Neumann problem: $\mathcal{L}_0 (u_0)=F$ in $\Omega$ and
$\partial u_0/\partial \nu_0=g$ on $\partial\Omega$ with the property $u_0 \perp \mathcal{R}$.
It follows from Remark \ref{remark-2.1} that
$$
\| u_\varep -\phi_\varep -u_0\|_{L^2(\Omega)}
\le C \, \varep^{1/2} 
\Big\{ \| g\|_{L^2(\partial\Omega)} + \| F\|_{L^2(\Omega)} \Big\}.
$$
By letting $w=u_0 +\phi_\varep$ this gives (\ref{5.1-0}).
\end{proof}

\begin{lemma}\label{lemma-5.2} 
Let $\varep\le r<1$.
Let $u_\varep \in H^1(D_{2r}; \rd)$ be a weak solution of
$\mathcal{L}_\varep (u_\varep)=F$ in $D_{2r}$ with $\partial u_\varep/\partial\nu_\varep=g$ on $\Delta_{2r}$.
Then there exists $w\in H^1(D_r; \rd)$ such that $\mathcal{L}_0 (w)=F$ in $D_r$,
$\partial w/\partial \nu_0=g$ on $\Delta_r$, and
\begin{equation}\label{5.2-0}
\aligned
&\left(\average_{D_r} |u_\varep -w|^2\right)^{1/2}\\
  &  \le C\left(\frac{\varep}{r} \right)^{1/2}
 \Bigg\{   \left(\average_{D_{2r}} |u_\varep|^2 \right)^{1/2}
  + r^2 \left(\average_{D_{2r}} |F|^2\right)^{1/2}
  + r\, \|  g\|_{L^\infty(\Delta_{2r})}  \Bigg\},
\endaligned
\end{equation}
where $C$ depends only on $d$, $\kappa_1$, $\kappa_2$, and $M$.
\end{lemma}

\begin{proof}
By rescaling we may assume $r=1$.
As in the case of Dirichlet  conditions in Lemma \ref{lemma-4.3},
the desired estimate follows from Lemma \ref{lemma-5.1}
by using the co-area formula and the following Cacciopoli's inequality 
\begin{equation}\label{Ca-N}
\int_{D_{3/2}} |\nabla u_\varep|^2
\le C \left\{ \int_{D_2} |u_\varep|^2 +\int_{D_2} |F|^2 +\| g\|^2_{L^\infty(\Delta_2)} \right\},
\end{equation}
where $\mathcal{L}_\varep (u_\varep)=F$ in $D_2$ and $\partial u_\varep/\partial\nu_\varep
=g$ on $\Delta_2$.
\end{proof}

\begin{lemma}\label{lemma-5.3}
Let $w$ be a solution of $\mathcal{L}_0 (w)=F$ in $D_r$ with $\partial w/\partial \nu_0=g$ on $\Delta_r$.
For $0<t\le r$, define
\begin{equation}\label{I}
\aligned
I(t; w)=  \frac{1}{t} \inf_{\substack{ M\in \mathbb{R}^{d\times d}\\ q\in \rd}} &
\Bigg\{  \left(\average_{D_t} |w- Mx -q|^2 \right)^{1/2}
+ t^2\left(\average_{D_t} |F|^p\right)^{1/p} \\ 
&+t\, \big\| \frac{\partial}{\partial\nu_0}
\big( w- Mx\big)\big \|_{L^\infty(\Delta_t)}
+t^{1+\sigma}\,  \big\|\frac{\partial}{\partial\nu_0}
\big(w-Mx\big)\big\|_{C^{0, \sigma} (\Delta_t)} \Bigg\},
\endaligned
\end{equation}
where $p>d$ and $\sigma\in (0, \alpha)$.
Then there exists $\theta\in (0,1/4)$, depending only on $d$, $p$, $\kappa_1$,
$\kappa_2$, $\sigma$, $\alpha$ and $M$, such that
\begin{equation}\label{5.3-0}
I(\theta r; w) \le (1/2) I(r; w).
\end{equation}
\end{lemma}

\begin{proof}
By rescaling we may assume $r=1$.
The lemma then follows from the boundary $C^{1, \sigma }$ estimates with Neumann boundary conditions
in $C^{1, \alpha}$ domains
for elasticity systems with constant coefficients.
\end{proof}

\begin{lemma}\label{lemma-5.4}
Let $0<\varep<1/2$.
Let $u_\varep$ be a solution of $\mathcal{L}_\varep (u_\varep)=F$ in $D_1$ with
 $\partial u_\varep/\partial\nu_\varep=g$ on $\Delta_1$,
 where $F\in L^p(D_1; \rd)$ for some $p>d$ and $g\in C^\sigma(\Delta_1;\rd)$ for some $\sigma\in (0, \alpha)$.
Define
\begin{equation}\label{J}
\aligned
J(r)=  \frac{1}{r} \inf_{\substack{ M\in \mathbb{R}^{d\times d}\\ q\in \rd}} &
\Bigg\{  \left(\average_{D_r} |u_\varep- Mx -q|^2 \right)^{1/2}
+ r^2\left(\average_{D_r} |F|^p\right)^{1/p} \\ 
&+ r\, \big\|  g-\frac{\partial}{\partial \nu_0} \big( Mx\big) \big\|_{L^\infty(\Delta_r)}
+r^{1+\sigma}\,  \big\| g-\frac{\partial}{\partial \nu_0}
\big(Mx\big)\big\|_{C^{0, \sigma} (\Delta_r)} \Bigg\},
\endaligned
\end{equation}
and
\begin{equation}\label{Psi}
\Psi (r)
=\frac{1}{r} \inf_{q\in \rd}
 \Bigg\{  \left(\average_{D_{2r}}   |u_\varep -q |^2 \right)^{1/2}
  + r^2 \left(\average_{D_{2r}} |F|^p\right)^{1/p}
+r\, \| g \|_{L^\infty(\Delta_{2r})}  \Bigg\}.
  \end{equation}
Then
\begin{equation}\label{5.4-0}
J(\theta r) \le (1/2) J(r) +C \left(\frac{\varep}{r}\right)^{1/2} \Psi (2r),
\end{equation}
for any $r  \in [\varep,1/2]$, where $\theta\in (0,1/4)$ is given by Lemma \ref{lemma-5.3}.
\end{lemma}

\begin{proof}
Fix $r\in [\varep, 1/2]$.
Let $w$ be the function in $H^1(D_r; \rd)$ given by Lemma \ref{lemma-5.2}.
Then
$$
\aligned
J(\theta r) &\le  I (\theta r; w) + \frac{1}{\theta r} \left(\average_{D_{\theta r}} |u_\varep -w|^2\right)^{1/2}\\
&\le  (1/2) I (r; w) + \frac{1}{\theta r} \left(\average_{D_{\theta r}} |u_\varep -w|^2\right)^{1/2}\\
& \le (1/2) J (r) + \frac{C}{r} \left(\average_{D_{r}} |u_\varep -w|^2\right)^{1/2},
\endaligned
$$
where we have used Lemma \ref{lemma-5.3} for the second inequality.
In view of Lemma \ref{lemma-5.2}, this gives
$$
J(\theta r)
\le (1/2) J(r)
+\frac{C}{r} 
 \Bigg\{  \left(\average_{D_{2r}}   |u_\varep |^2 \right)^{1/2}
  + r^2 \left(\average_{D_{2r}} |F|^p\right)^{1/p}
+r\, \| g \|_{L^\infty(\Delta_{2r})}  \Bigg\},
$$
from which the estimate (\ref{5.4-0}) follows, as the function $J(r)$ is invariant if we replace $u_\varep$
by $u_\varep -q$ for any $q\in \rd$.
\end{proof}

\begin{proof}[\bf Proof of Theorem \ref{main-theorem-4}]
With Lemma \ref{lemma-5.4} at our disposal,
 Theorem \ref{main-theorem-4} follows from Lemma \ref{G-lemma},
as in the case of Dirichlet boundary conditions.
We omit the details.
\end{proof}

As we indicate in the Introduction, under additional smoothness conditions,
the full Lipschitz estimates, uniform in $\varep$,
 follow from Theorem \ref{main-theorem-3}, Theorem \ref{main-theorem-4},
 and local Lipschitz estimates by a blow-up argument.

\begin{corollary}\label{cor-5.0}
Suppose that $A$ satisfies (\ref{ellipticity})-(\ref{periodicity}).
Also assume that $A$ is H\"older continuous.
Let $u_\varep\in H^1(B(0,1); \rd)$ be a weak solution of $\mathcal{L}_\varep (u_\varep)
=F$ in $B(0,1)$, where $F\in  L^p(B(0,1); \rd)$ for some $p>d$.
Then 
\begin{equation}\label{Lip-1}
\|\nabla u_\varep\|_{L^\infty(B(0,1/2))}
\le C _p\Big\{ \| u_\varep\|_{L^2(B(0,1))}
+\| F \|_{L^p(B(0,1))} \Big\},
\end{equation}
where $C_p$ depends only on $d$, $p$ and $A$.
\end{corollary}

\begin{corollary}\label{cor-5.1}
Suppose that $A$ satisfies  (\ref{ellipticity})-(\ref{periodicity}).
Also assume that $A$ is H\"older continuous.
Let $u_\varep \in H^1(D_1; \rd)$ be a weak solution of $\mathcal{L}(u_\varep)=F$ in $D_1$
with $u_\varep =f$ on $\Delta_1$,
where the defining function $\psi$ in $D_1$ and $\Delta_1$ is $C^{1, \alpha}$
with $\|\nabla \psi\|_{C^\alpha(\mathbb{R}^{d-1})} \le M$
for some $\alpha>0$. Then 
\begin{equation}\label{Lip-2}
\| \nabla u_\varep\|_{L^\infty(D_{1/2})}
\le C \Big\{ \| u_\varep\|_{L^2(D_1)} +\| F\|_{L^p(D_1)} +\| f\|_{C^{1,\sigma}(\Delta_1)} \Big\},
\end{equation}
where $p>d$, $\sigma\in (0, \alpha)$, and $C$ depends only on $d$, $p$, $\sigma$, $A$, $\alpha$ and $M$.
\end{corollary}

\begin{corollary}\label{cor-5.2}
Suppose that $A$, $D_1$ and $\Delta_1$ satisfy the same conditions as in Corollary \ref{cor-5.1}.
Let $u_\varep \in H^1(D_1; \rd)$ be a weak solution of $\mathcal{L}(u_\varep)=F$ in $D_1$
with $\frac{\partial u_\varep}{\partial \nu_\varep} =g$ on $\Delta_1$.
Then 
\begin{equation}\label{Lip-3}
\| \nabla u_\varep\|_{L^\infty(D_{1/2})}
\le C \Big\{ \| u_\varep\|_{L^2(D_1)} +\| F\|_{L^p(D_1)} +\| g\|_{C^{\sigma}(\Delta_1)} \Big\},
\end{equation}
where $p>d$, $\sigma\in (0, \alpha)$, and $C$ depends only on $d$, $p$, $\sigma$, $A$, 
 $\alpha$ and $M$.
\end{corollary}

As we mentioned in Introduction, for $\mathcal{L}_\varep$ with coefficients satisfying (\ref{s-ellipticity}),
(\ref{periodicity}) and the H\"older continuity condition,
estimates (\ref{Lip-1}) and (\ref{Lip-2}) were proved in \cite{AL-1987},
while (\ref{Lip-3}) was established in \cite{KLS1, Armstrong-Shen-2014}.

\bibliography{s37.bbl}

\medskip

\begin{flushleft}
Zhongwei Shen

 Department of Mathematics
 
University of Kentucky

Lexington, Kentucky 40506,
USA. 


E-mail: zshen2@uky.edu
\end{flushleft}

\medskip

\noindent \today

\end{document}